\documentclass[11pt]{amsart}
\usepackage{amsmath,amsfonts,amsthm}
\usepackage{amssymb,latexsym}
\usepackage[colorlinks]{hyperref}
\usepackage{graphicx}
\usepackage[all]{xy}
\usepackage{layout}
\theoremstyle{plain}
\newtheorem{theorem}{Theorem}[section]
\newtheorem{lemma}[theorem]{Lemma}
\newtheorem{proposition}[theorem]{Proposition}
\newtheorem{corollary}[theorem]{Corollary}

\theoremstyle{definition}

\newtheorem{definition}[theorem]{Definition}
\newtheorem{example}[theorem]{Example}

\theoremstyle{remark}
\newtheorem*{notation}{Notation}
\newtheorem*{remark}{Remark}
\numberwithin{equation}{section}

\newcommand{\N}{\mathbb{N}}

\newcommand{\C}{\mathbb{C}}
\newcommand{\Prime}{\mathbb{P}}

\title{The category of Bratteli diagrams}

\author{Massoud Amini,  George A. Elliott, and Nasser Golestani}

\address[\textbf{Massoud Amini}]{Department of Pure Mathematics\\ Faculty of Mathematical Sciences\\
 University of Tarbiat Modares\\
 Tehran\\ Iran}
\email {mamini@modares.ac.ir }

\address[\textbf{George A. Elliott}]{Department of Mathematics\\ University of Toronto\\ Toronto, Ontario, Canada\ \ M5S 2E4}
\email {elliott@math.toronto.edu}

\address[\textbf{Nasser Golestani}]{Department of Pure Mathematics\\ Faculty of Mathematical Sciences\\
 University of Tarbiat Modares\\
 Tehran\\ Iran}
\email{n.golestani@modares.ac.ir}
\begin{document}
\begin{abstract}
A category structure for Bratteli diagrams is proposed and a functor from
the category of AF algebras to the category of Bratteli diagrams is
constructed. Since  isomorphism of Bratteli diagrams in this category coincides
with Bratteli's notion of equivalence, we obtain in particular a functorial formulation of Bratteli's
classification of AF algebras (and at the same time, of Glimm's classification of UHF~algebras).
It is shown that the three approaches
to classification of AF~algebras, namely, through Bratteli diagrams, K-theory, and
abstract classifying categories, are essentially the same
from a categorical point of view.
\end{abstract}
\maketitle
\section{Introduction}
\pagestyle{plain}
\noindent AF algebras were first introduced and studied by Bratteli in 1972 \cite{br72}. An
AF algebra  is a C$^{*}$-algebra which is the closure of the union  of an
increasing sequence of its finite dimensional C$^{*}$-subalgebras.
The class of AF algebras has an interesting variety of examples \cite{br72, da96}.
AF algebras are generalizations of UHF~algebras which were
studied by Glimm in 1960 \cite{gl60} and of matroid C$^{*}$-algebras (stably isomorphic to UHF~algebras)
introduced by Dixmier in 1967 \cite{di67}.
Glimm gave a classification of UHF~algebras. In a brilliant leap, Bratteli generalized Glimm's classification to arbitrary AF~algebras (see below---Theorem~\ref{thrst} is a reformulation of this).

In 1976, Elliott gave a  classification of
 AF algebras using K-theory \cite{el76}.
 In fact, Elliott showed that the functor
 $\mathrm{K}_{0}: \mathbf{AF}\to \mathbf{DG}$, from the category of
 AF~algebras with  $*$-homomorphisms
 to the category of (scaled countable) dimension groups with order-preserving homomorphisms,
is a strong classification functor, in the sense that
if $\mathcal{A}_{1},\mathcal{A}_{2}\in \mathbf{AF}$ and
$\mathrm{K}_{0}(\mathcal{A}_{1})\cong \mathrm{K}_{0}(\mathcal{A}_{2})$, then we have
$\mathcal{A}_{1}\cong \mathcal{A}_{2}$, and in fact every isomorphism from
$\mathrm{K}_{0}(\mathcal{A}_{1})$ onto $\mathrm{K}_{0}(\mathcal{A}_{2})$ comes from an isomorphism from $\mathcal{A}_{1}$ onto $\mathcal{A}_{2}$
 (see \cite{el76}, \cite[Section~7.2]{mu90}, and \cite[Sections 5.1--5.3]{el10} for details).
 This categorical idea,  finding a (strong) classification functor from a given category to another, more accessible category,
 is useful in the classification of various categories (see \cite{el10}).

The classification of AF algebras obtained by Bratteli in \cite{br72} used
what are now called Bratteli diagrams.
Bratteli  associated to each AF algebra $\mathcal{A}$ an infinite
directed graph $\mathcal{B}(\mathcal{A})$,  its Bratteli diagram (see Definition~\ref{defbd}),
and used these very effectively to study AF algebras.
Some attributes of an AF algebra can be read directly from its Bratteli diagram,
for instance its ideal structure.
Bratteli showed that  for
$\mathcal{A}_{1},\mathcal{A}_{2}\in \mathbf{AF}$,
$\mathcal{A}_{1}\cong \mathcal{A}_{2}$ if
$\mathcal{A}_{1}$ and $\mathcal{A}_{2}$ have the same Bratteli diagram,
i.e.,
$\mathcal{B}(\mathcal{A}_{1})=\mathcal{B}(\mathcal{A}_{2})$
(see Theorem~\ref{olab}). In fact Bratteli determined, in terms of the Bratteli diagrams of $\mathcal{A}_{1}$ and $\mathcal{A}_{2}$, exactly when $\mathcal{A}_{1}$ and $\mathcal{A}_{2}$ are isomorphic.

Denote by $\mathbf{BD}$ the set of all Bratteli diagrams. Then, Bratteli's theorem  asserts that the map
$\mathcal{B}: \mathbf{AF}\to \mathbf{BD}$ has the property that if
$\mathcal{A}_{1},\mathcal{A}_{2}\in \mathbf{AF}$ and
$\mathcal{B}(\mathcal{A}_{1})=\mathcal{B}(\mathcal{A}_{2})$, or even just $\mathcal{B}(\mathcal{A}_{1})$ is equivalent in Bratteli's sense to $\mathcal{B}(\mathcal{A}_{2})$,  then
$\mathcal{A}_{1}\cong \mathcal{A}_{2}$.  The question that arises naturally here is whether the
map $\mathcal{B}: \mathbf{AF}\to \mathbf{BD}$ can be made into a functor,
and if so, whether it is a  classification functor. This
paper answers these questions.

In Section~2 we define an appropriate notion of  morphism in $\mathbf{BD}$ and we show that
$\mathbf{BD}$ with these morphisms is a category (Theorem~\ref{thrbd}).
In Section~3 we show that $\mathcal{B}: \mathbf{AF}\to \mathbf{BD}$
is a (strong) classification functor (Theorem~\ref{thrst}); thus for
$\mathcal{A}_{1},\mathcal{A}_{2}\in \mathbf{AF}$
we have $\mathcal{A}_{1}\cong \mathcal{A}_{2}$ if, and only if,
$\mathcal{B}(\mathcal{A}_{1})\cong\mathcal{B}(\mathcal{A}_{2})$.
This is a functorial formulation of Bratteli's theorem and would appear to be a definitive  elaboration of  the
classification of AF~algebras from the Bratteli diagram point of view.
In particular,  just the fact that the map is a functor yields
Glimm's classification theorem for UHF~algebras (see the proof of Theorem~\ref{thrgl}).

In Section~4, it is shown that the functor $\mathcal{B}: \mathbf{AF}\to \mathbf{BD}$
is a full functor (Theorem~\ref{thrful}),
which means that homomorphisms in the codomain category can be lifted back to homomorphisms in the domain category (this was done for isomorphisms in Theorem~\ref{thrst}).

In Section~5, we investigate the relation between the category  $\mathbf{BD}$ of Bratteli diagrams and two  abstract classifying
categories, $\mathbf{AF}^{\mathrm{out}}$ and $\overline{\mathbf{AF}^{\mathrm{out}}}$, for AF~algebras (cf.~\cite{el10}).
We show that there is a strong classification  functor from
$\overline{\mathbf{AF}^{\mathrm{out}}}$ to $\mathbf{BD}$
which is faithful and full (Theorem~\ref{throutbar}) and is an equivalence of categories
(Theorem~\ref{thrfung}).

In Section~6, we investigate the relation between
$\overline{\mathbf{AF}^{\mathrm{out}}}$
and the category $\mathbf{DG}$ of dimension groups.
We show that there is a strong classification  functor from
$\overline{\mathbf{AF}^{\mathrm{out}}}$ to $\mathbf{DG}$
which is faithful and full (Theorem~\ref{thrkbar}) and is an equivalence of categories
(Theorem~\ref{thrfung0}). It is shown that the three strong classification functors $\mathcal{B}:\mathbf{AF}\to \mathbf{BD}$,
$\overline{\mathcal{F}} :\mathbf{AF}\to  \overline{\mathbf{AF}^{\mathrm{out}}}$, and
$\mathrm{K}_{0}: \mathbf{AF}\to \mathbf{DG}$ which  classify
 AF~algebras are essentially the same (Theorem~\ref{thrfuns}).
\section{The Category of Bratteli Diagrams $\mathbf{BD}$}
\noindent The notion of a Bratteli diagram was  introduced by
Bratteli to study AF algebras \cite{br72}. There are various formal definitions
(just with different formulations)
for a Bratteli diagram; for example see  \cite{du99} and \cite{lt80}.
What is behind these definitions is the very special structure of a $*$-homomorphism
between finite dimensional C$^{*}$-algebras. The following theorem of
Bratteli describes this structure \cite{br72}. Let us just quote this theorem,  with some slight changes, from \cite{da96}.
\begin{theorem}\label{thrfd}
With $A_{1}$ and $A_{2}$ denoting the finite dimensional C$^{*}$-algebras $\mathcal{M}_{n_{1}}\oplus\cdots \oplus\mathcal{M}_{n_{k}}$ and
 $\mathcal{M}_{m_{1}}\oplus\cdots \oplus\mathcal{M}_{m_{l}}$ (where the n's and m's are non-zero), let
 $\varphi :A_{1}\to A_{2}$ be a $\ast$-homomorphism. Then there is a unique $l\times k$
 matrix $E=(a_{ij})$ of positive (i.e., non-negative) integers with the property that there is a
 unitary $v= (v_{1},\ldots, v_{l})$ in $A_{2}$ such that
 if we set $\varphi_{i}=\pi_{i}\circ \varphi :A_{1}\to \mathcal{M}_{m_{i}}$ then

 \begin{equation*}
 v_{i}\varphi_{i}(u_{1},\ldots, u_{k})v_{i}^{*}=
 \left(
 \begin{smallmatrix}
  \resizebox{.9cm}{.4cm}{$u_{1}^{(a_{i1})}$} &  & & \resizebox{.2cm}{.22cm}{$0$}\ \ \ \ \\
  & \ddots &  & \\
  &  & \resizebox{.9cm}{.4cm}{$u_{k}^{(a_{ik})}$} &  \\
 \resizebox{.2cm}{.22cm}{$0$} & & & \resizebox{.67cm}{.31cm}{$0^{(s_{i})}$}
 \end{smallmatrix}
 \right), \
  \quad\  (u_{1},\ldots, u_{k})\in A_{1},
 \end{equation*}

 \noindent where $s_{i}$ is defined by the equation
 $\sum_{j=1}^{k}a_{ij}n_{j}+ s_{i}= m_{i}$, for each $1\leq
i\leq l$. Thus, if $V_{1}$ and $V_{2}$ denote the column matrices such that
$V_{1}^{T}=(n_{1} \cdots \ n_{k})$ and
 $V_{2}^{T}=(m_{1} \cdots \ m_{l})$, then $EV_{1}\leq V_{2}$.
 Moreover, we have
 \begin{itemize}
 \item[(1)] $\varphi$ is injective if and only if for each $j$ there is an $i$ such that
 $a_{ij}\neq 0$;
 \item[(2)] $\varphi$ is unital if and only if $EV_{1}= V_{2}$, i.e.,  $s_{i}=0$ for
 each $1\leq i\leq l$.
 \end{itemize}
\end{theorem}
\begin{proof}
See \cite[Proposition 1.7]{br72} and \cite[Corollary III.2.2]{da96}.
\end{proof}
Let us call the matrix $E$ in the previous theorem the  multiplicity matrix of $\varphi$, and denote it by
$R_{\varphi}$ (this is the notation used in \cite{bl98}). In general, let $V_{i}$  be a $k_{i}\times 1$ matrix of
non-zero positive integers for $i=1,2$;
by a  \emph{multiplicity matrix} $E=(a_{ij})$ from $V_{1}$ to $V_{2}$ we shall mean  a $k_{2}\times k_{1}$ matrix of positive integers such that $EV_{1}\leq V_{2}$.
We shall use the notation $E: V_{1}\to  V_{2}$  to mean that $E$ is a
multiplicity matrix from $V_{1}$ to $V_{2}$.
$E$ will be called an \emph{embedding matrix} if for each $j$ there is an $i$ such that
 $a_{ij}\neq 0$ (in other words, if the algebra map induced by $E$---as defined above---, is injective).

Let us recall the   formulation of the  definition of a Bratteli diagram in \cite[Sections 2 and 3]{el10},
which uses the matrix language and is more flexible for our purposes.
\begin{definition}\label{defbd}
By a \emph{Bratteli diagram} let us mean an ordered pair $B=(V,E)$,
$V=(V_{n})_{n=1}^{\infty}$ and $E=(E_{n})_{n=1}^{\infty}$, such that:
\begin{itemize}
\item[(1)]
each $V_{n}$ is a $k_{n}\times 1$ matrix of non-zero positive integers for
some
\end{itemize}
\noindent
$k_{n}\geq 1$;
\begin{itemize}
\item[(2)]
each $E_{n}$ is an embedding matrix  from $V_{n}$ to $V_{n+1}$.
\end{itemize}
 Let us denote such a $B$ by the  diagram

\[
\xymatrix{V_{1}\ar[r]^-{\resizebox{1.15em}{.65em}{$E_{1}$}}
 &V_{2}\ar[r]^-{\resizebox{1.15em}{.65em}{$E_{2}$}}
 &V_{3}\ar[r]^-{\resizebox{1.15em}{.65em}{$E_{3}$}}
 &\cdots\ .
 }
\]
\end{definition}
Let us write $E_{nm}=E_{m-1}\cdots E_{n+1}E_{n}$ for $n< m$ and  $E_{nn}=I$, where $I$ is the identity matrix of order $k_{n}$. Note that $E_{nm}$ is a  multiplicity matrix from $V_{n}$ to $V_{m}$.
\begin{remark}
 In Definition~\ref{defbd}, we have in fact defined the notion of a ``non-zero" Bratteli diagram.
 This is enough for working with non-zero (in particular, unital) AF~algebras.
 For the zero AF~algebra, we get the  zero Bratteli diagram, which is nothing
but a single zero square matrix of size one,
denoted by $0$.
\end{remark}
Let $\mathbf{BD}$ denote the set of all Bratteli diagrams. We wish to define morphisms between objects in $\mathbf{BD}$ to make it a category. In order to formulate the correct notion of morphism for our purposes, we first need to define the notion of premorphism. Recall
that a sequence $(f_{n})_{n=1}^{\infty}$ of positive integers
 is said to be cofinal in $\N$ if $\sup\{f_{n}\, |\, n\in \N\}=+\infty$.
\begin{definition}\label{defpre}
Let $B=(V,E)$ and $C=(W,S)$ be  Bratteli diagrams. A
\emph{premorphism} $f:B \to C$ is an ordered pair
$((F_{n})_{n=1}^{\infty},(f_{n})_{n=1}^{\infty})$ where
$(F_{n})_{n=1}^{\infty}$ is a sequence of matrices  and $(f_{n})_{n=1}^{\infty}$
a cofinal sequence of positive integers
with $f_{1}\leq f_{2}\leq \cdots$ such that:
\begin{itemize}
\item[(1)] each $F_{n}$ is a  multiplicity  matrix from $V_{n}$ to $W_{f_{n}}$;
\item[(2)] the diagram of $f:B\to C$  commutes:

\[
\xymatrix{V_{1}\ar[r]^{E_{1}}\ar[d]_{F_{1}}
 &V_{2}\ar[r]^-{E_{2}}\ar[d]_{F_{2}} &V_{3}\ar[r]^-{E_{3}}\ar[ld]^{F_{3}} &\cdots \ \ \  \\
 W_{1}\ar[r]_{S_{1}}
 &W_{2}\ar[r]_{S_{2}} &W_{3}\ar[r]_{S_{3}}&\cdots \   .
 }
\]
\end{itemize}
Commutativity of the diagram of course amounts to saying that for any positive integer $n$  we have  $F_{n+1}E_{n}=S_{f_{n}f_{n+1}}F_{n}$; that is, the square

\[
\xymatrix{V_{n}\ar[r]^{E_{n}}\ar[d]_{F_{n}}
 & V_{n+1}\ar[d]^{F_{n+1}}  \\
 W_{f_{n}}\ar[r]_{S_{f_{n}f_{n+1}}}
 & W_{f_{n+1}}
 }
\]

\noindent commutes. (This implies the general property of commutativity, namely, that any two paths of maps between the same pair of points in the diagram agree, i.e., have the same product.)

Let $B$, $C$, and $D$ be objects in $\mathbf{BD}$ and let $f:B\to C$ and $g:C\to D$ be premorphisms,
$f=((F_{n})_{n=1}^{\infty},(f_{n})_{n=1}^{\infty})$ and
$g=((G_{n})_{n=1}^{\infty},(g_{n})_{n=1}^{\infty})$. The composition of $f$ and $g$ is defined as $gf=((H_{n})_{n=1}^{\infty},(h_{n})_{n=1}^{\infty})$, where
$H_{n}=G_{f_{n}}F_{n}$ and $h_{n}=g_{f_{n}}$.
\end{definition}
\begin{remark}
In Definition~\ref{defpre}, it is implicitly assumed that
the Bratteli diagrams in question are non-zero. Let us
define the zero premorphism as follows.
Let $B$ be a Bratteli diagram.
The zero premorphism from $B$ to $0$ (the zero Bratteli diagram) is
the ordered pair $(B,0)$. Similarly, the zero premorphism from $0$ to $B$
is $(0,B)$. (Note that a morphism in a category depends on both
the domain and the codomain objects.) The composition of
the zero premorphism with any other premorphism is defined to be zero.
\end{remark}
\begin{proposition}\label{propre}
The set $\mathbf{BD}$, with premorphisms as defined above as maps, is a
(small) category.
\end{proposition}
\begin{proof}
First let us check that if $f:B\to C$ and $g:C\to D$ are premorphisms, with
$f=((F_{n})_{n=1}^{\infty},(f_{n})_{n=1}^{\infty})$ and
$g=((G_{n})_{n=1}^{\infty},(g_{n})_{n=1}^{\infty})$, then $gf$ given as above, by
$H_{n}=G_{f_{n}}F_{n}$ and $h_{n}=g_{f_{n}}$, is  a premorphism.
Write $B=(V,E)$, $C=(W,J)$, and $D=(Z,T)$.

Let $n$ be a positive integer  and consider the following diagram:

\[
\xymatrix{V_{n}\ar[r]^{E_{n}}\ar[d]_{F_{n}}
 & V_{n+1}\ar[d]^{F_{n+1}}  \\
 W_{f_{n}}\ar[r]_{S_{f_{n}f_{n+1}}}\ar[d]_{G_{f_{n}}}
 & W_{f_{n+1}}\ar[d]^{ G_{f_{n+1}}} \\
 Z_{h_{n}}\ar[r]_{T_{h_{n}h_{n+1}}}  & Z_{h_{n+1}}\ .
 }
\]

\noindent Since $f$ and $g$ are premorphisms we have
\begin{align}
T_{h_{n+1}h_{n}}H_{n} &=
T_{h_{n+1}h_{n}}G_{f_{n}}F_{n}\notag\\
&= G_{f_{n+1}}S_{f_{n}f_{n+1}}F_{n}\notag\\
&= G_{f_{n+1}}F_{n+1}E_{n}\notag\\
&=H_{n+1}E_{n}. \notag
\end{align}
This shows that  $gf$ is indeed a premorphism.

Now for any Bratteli diagram $B=(V,E)$ define the premorphism $\mathrm{id}_{B}:B\to B$ by
$\mathrm{id}_{B}=((I_{n})_{n=1}^{\infty},(i_{n})_{n=1}^{\infty})$, where
$I_{n}$ is the identity matrix of order  equal to the number of columns of
$V_{n}$, and $i_{n}=n$. For any premorphisms  $f:B\to C$ and $h:C\to B$ we have $\mathrm{id}_{B}h=h$ and $f\,\mathrm{id}_{B}=f$.

One checks easily that  associativity holds using the associativity  of matrix multiplication. This completes the proof that $\mathbf{BD}$, with premorphisms, is a category.
\end{proof}
It will be clear later that the category $\mathbf{BD}$ with  premorphisms is
not suitable for the classification of AF algebras and that we need to consider morphisms---consisting of equivalence classes of premorphisms---for the purposes of classification.
\begin{definition}\label{defeq}
Let $B, C$ be Bratteli diagrams, and $f,g: B\to C$ be premorphisms, i.e., maps in the category $\mathbf{BD}$ of Proposition~\ref{propre}, with $B=(V,E)$, $C=(W,S)$,
$f=((F_{n})_{n=1}^{\infty},(f_{n})_{n=1}^{\infty})$, and
$g=((G_{n})_{n=1}^{\infty},(g_{n})_{n=1}^{\infty})$. Let us  say that $f$ is \emph{equivalent} to $g$, and write $f\sim g$, if there are sequences $(n_{k})_{k=1}^{\infty}$ and $(m_{k})_{k=1}^{\infty}$ of
positive integers such that $n_{k}<m_{k}<n_{k+1}$ and $f_{n_{k}}<g_{m_{k}}<f_{n_{k+1}}$
for each $k\geq 1$, and the diagram

\[
\xymatrix{V_{n_{\resizebox{.28em}{.28em}{1}}}
\ar[r]\ar[d]_{F_{n_{\resizebox{.26em}{.26em}{1}}}}
 & V_{m_{\resizebox{.28em}{.28em}{1}}}
 \ar[r]\ar[d]_{G_{m_{\resizebox{.26em}{.26em}{1}}}}  &
 V_{n_{\resizebox{.28em}{.28em}{2}}}
 \ar[r] \ar[d]_{F_{n_{\resizebox{.26em}{.26em}{2}}}}
 & V_{m_{\resizebox{.28em}{.28em}{2}}}
 \ar[r]\ar[d]_{G_{m_{\resizebox{.26em}{.26em}{2}}}} &\cdots \\
 W_{f_{n_{\resizebox{.26em}{.26em}{1}}}}\ar[r]
 &W_{g_{m_{\resizebox{.26em}{.26em}{1}}}}\ar[r] &
 W_{f_{n_{\resizebox{.26em}{.26em}{2}}}}\ar[r]
 & W_{g_{m_{\resizebox{.26em}{.26em}{2}}}}\ar[r]&\cdots
 }
\]

\noindent commutes, i.e., each minimal square commutes: for each $k\geq 1$,
\begin{equation*}
G_{m_{k}}E_{n_{k}m_{k}}=S_{f_{n_{k}}g_{m_{k}}}F_{n_{k}}\ \
\mathrm{and}\ \  F_{n_{k+1}}E_{m_{k}n_{k+1}}=S_{g_{m_{k}}f_{n_{k+1}}}G_{m_{k}}.
\end{equation*}
\end{definition}
At the end of this section we will give two other definitions for equivalence of
premorphisms (Definition~\ref{defeq2} and Definition~\ref{defeq3}).
These may be more natural in some sense (since they do not use subsequences), but we shall
show that all three equivalence relations are the same (Proposition~\ref{proeqpre}).
\begin{remark}
In Definition~\ref{defeq}, we have defined the equivalence of a pair
of non-zero premorphisms. This notion extends in an obvious way to encompass
the zero premorphism (which is defined in the remark following Definition~\ref{defpre}),
since for each Bratteli diagram $B$,
$\mathrm{Hom}(B,0)$ and $\mathrm{Hom}(0,B)$ have only one element.
\end{remark}
\begin{lemma}\label{lemeqpre}
Let $B,C\in \mathbf{BD}$. Then $\sim$ is an equivalence relation on the set of
 premorphisms from $B$ to $C$.
\end{lemma}
\begin{proof}
It is obvious that $\sim$ is symmetric. Reflexivity follows from Definition~\ref{defpre}, using
the cofinality condition. Let $f,g,h:B\to C$ be premorphisms such that $f\sim g$ and $g\sim h$, and let us show
that $f\sim h$.
Write $B=(V,E)$, $C=(W,S)$,
$f=((F_{n})_{n=1}^{\infty},(f_{n})_{n=1}^{\infty})$,
$g=((G_{n})_{n=1}^{\infty},(g_{n})_{n=1}^{\infty})$, and
$h=((H_{n})_{n=1}^{\infty},(h_{n})_{n=1}^{\infty})$.

Choose sequences $(n_{k})_{k=1}^{\infty}$ and $(m_{k})_{k=1}^{\infty}$ establishing $f\sim g$,
according to Definition~\ref{defeq}, and  sequences
$(p_{k})_{k=1}^{\infty}$ and $(q_{k})_{k=1}^{\infty}$ for $g\sim h$.
Construct sequences $(r_{k})_{k=1}^{\infty}$ and $(s_{k})_{k=1}^{\infty}$
inductively as follows to show that $f\sim h$.
Set $n_{1}=r_{1}$. There is $k_{0}\geq 1$ such that $p_{k_{0}}> m_{1}$;
set $q_{k_{0}}=s_{1}$. Each square in the  diagram

\[
\xymatrix{V_{r_{\resizebox{.28em}{.28em}{1}}}\ar[r]
\ar[d]_{F_{r_{\resizebox{.26em}{.26em}{1}}}}
 & V_{m_{\resizebox{.28em}{.28em}{1}}}\ar[r]
 \ar[d]_{G_{m_{\resizebox{.26em}{.26em}{1}}}}  &
 V_{p_{k_{\resizebox{.26em}{.26em}{0}}}}\ar[r]
 \ar[d]_{G_{p_{k_{\resizebox{.26em}{.26em}{0}}}}}
 & V_{s_{\resizebox{.28em}{.28em}{1}}}
 \ar[d]_{H_{s_{\resizebox{.26em}{.26em}{1}}}}  \\
 W_{f_{r_{\resizebox{.26em}{.26em}{1}}}}\ar[r]
 &W_{g_{m_{\resizebox{.26em}{.26em}{1}}}}\ar[r] &
 W_{g_{p_{k_{\resizebox{.26em}{.26em}{0}}}}}\ar[r]
 &W_{h_{s_{\resizebox{.26em}{.26em}{1}}}}
 }
\]

\noindent  commutes, by the definitions of $f\sim g$ and
$g\sim h$, and since $g$ is a premorphism. Thus,
$H_{s_{1}}E_{r_{1}s_{1}}= S_{f_{r_{1}}h_{s_{1}}} F_{r_{1}}$.

There is $k_{1}\geq 1$ such that $m_{k_{1}}>p_{k_{0}+1}$; set $n_{k_{1}+1}=r_{2}$.
 In the diagram

\[
\xymatrix{V_{s_{\resizebox{.28em}{.28em}{1}}}\ar[r]
\ar[d]_{H_{s_{\resizebox{.26em}{.26em}{1}}}}
 & V_{p_{k_{\resizebox{.26em}{.26em}{0}}+1}}\ar[r]
 \ar[d]_{G_{p_{k_{\resizebox{.26em}{.26em}{0}}+1}}}  &
 V_{m_{k_{\resizebox{.26em}{.26em}{1}}}}\ar[r]
 \ar[d]_{G_{m_{k_{\resizebox{.26em}{.26em}{1}}}}}
 & V_{r_{\resizebox{.28em}{.28em}{2}}}
 \ar[d]_{F_{r_{\resizebox{.26em}{.26em}{2}}}}  \\
 W_{h_{s_{\resizebox{.26em}{.26em}{1}}}}\ar[r]
 &W_{g_{p_{k_{\resizebox{.26em}{.26em}{0}}+1}}}\ar[r] &
 W_{g_{m_{k_{\resizebox{.26em}{.26em}{1}}}}}\ar[r]
 &W_{f_{r_{\resizebox{.26em}{.26em}{2}}}}\ ,
 }
\]

\noindent each square commutes, and so we have
$F_{r_{2}}E_{s_{1}r_{2}}= S_{h_{s_{1}}f_{r_{2}}}H_{s_{1}} $.

Continuing this procedure we obtain sequences
$(r_{k})_{k=1}^{\infty}$ and $(s_{k})_{k=1}^{\infty}$ with
$r_{1}<s_{1}<r_{2}<s_{2}<\cdots$ and
$f_{r_{1}}<h_{s_{1}}<f_{r_{2}}<h_{s_{2}}<\cdots$ such that commutativity holds as required
in Definition~\ref{defeq} for $f\sim h$. This shows that
$\sim$ is transitive, and  so it is an equivalence relation.
\end{proof}
Let us call an equivalence class of premorphisms between Bratteli diagrams $B$ and $C$,
with respect to the relation $\sim$, a \emph{morphism} from $B$ to $C$.
Let us denote the equivalence class of a premorphism
$f:B\to C$  by $[f]:B\to C$, or if there is  no confusion, just   by $f$.

The equivalence class of the zero premorphism (which makes sense only
when $B=0$ or $C=0$) is called the zero morphism
(see the remark preceding Lemma~\ref{lemeqpre}).

The composition of morphisms $[f]:B\to C$ and $[g]:C\to D$ should defined as
$[gf]:B\to D$ where $gf$ is the composition of premorphisms. This composition is
well~defined, as is shown in the proof of the following theorem.
\begin{theorem}\label{thrbd}
The set $\mathbf{BD}$, with  morphisms as defined above, is a category.
\end{theorem}
\begin{proof}
First, we must show that the composition of two morphisms is well defined, i.e.,  independent of the choice of representatives.
Let $f,l:B\to C$ and $g,h:C\to D$ be premorphisms such that
$f\sim l$ and $g\sim h$, and let us show that $gf\sim hl$.

Write $B=(V,E)$, $C=(W,S)$, and $D=(Z,T)$, and
 $f=((F_{n})_{1}^{\infty}, (f_{n})_{n=1}^{\infty})$,
$g=((G_{n})_{n=1}^{\infty}, (g_{n})_{1}^{\infty})$, and
$gf=((U_{n})_{n=1}^{\infty}, (u_{n})_{1}^{\infty})$. Then $u_{n}=g_{f_{n}}$ and
$U_{n}=G_{f_{n}}F_{n}$, according to Definition~\ref{defpre}.
Also write $h=((H_{n})_{1}^{\infty}, (h_{n})_{n=1}^{\infty})$,
$l=((L_{n})_{n=1}^{\infty}, (l_{n})_{1}^{\infty})$, and
$hl=((X_{n})_{n=1}^{\infty}, (x_{n})_{1}^{\infty})$. Then $x_{n}=h_{l_{n}}$ and
$X_{n}=H_{l_{n}}L_{n}$, according to Definition~\ref{defpre}.

Let $(n_{k})_{k=1}^{\infty}$ and
$(m_{k})_{k=1}^{\infty}$ be sequences exhibiting the equivalence $f\sim l$, and $(p_{k})_{k=1}^{\infty}$
and $(q_{k})_{k=1}^{\infty}$ sequences exhibiting the equivalence $g\sim h$, according to Definition~\ref{defeq}.
Let us construct  sequences $(r_{k})_{k=1}^{\infty}$ and
$(s_{k})_{k=1}^{\infty}$ exhibiting the equivalence $gf\sim hl$. Set $n_{1}=r_{1}$. There is $k_{0}\geq 1$
such that $p_{k_{0}}>f_{n_{1}}$. There is $k_{1}\geq 1$ such that $l_{m_{k_{1}}}\geq q_{k_{0}}$;
set $s_{1}=m_{k_{1}}$.
Consider the diagram

\[
\xymatrix{
W_{f_{r_{\resizebox{.26em}{.26em}{1}}}}\ar[r]
\ar[d]_{G_{f_{r_{\resizebox{.26em}{.26em}{1}}}}}
 & W_{p_{k_{\resizebox{.26em}{.26em}{0}}}}
 \ar[r]\ar[d]_{G_{p_{k_{\resizebox{.26em}{.26em}{0}}}}}
 & W_{q_{k_{\resizebox{.26em}{.26em}{0}}}}\ar[r]
 \ar[d]_{H_{q_{k_{\resizebox{.26em}{.26em}{0}}}}}
 & W_{l_{s_{\resizebox{.26em}{.26em}{1}}}}
 \ar[d]_{H_{l_{s_{\resizebox{.26em}{.26em}{1}}}}}  \\
 Z_{g_{f_{r_{\resizebox{.26em}{.26em}{1}}}}}\ar[r]
 &Z_{g_{p_{k_{\resizebox{.26em}{.26em}{0}}}}}\ar[r] &
 Z_{h_{q_{k_{\resizebox{.26em}{.26em}{0}}}}}\ar[r]
 &Z_{h_{l_{s_{\resizebox{.26em}{.26em}{1}}}}}\ .
 }
 \]

\noindent Each square in this diagram commutes, because $g$ and $h$ are
premorphisms and $g\sim h$. Thus
$H_{l_{s_{1}}}S_{f_{r_{1}}l_{s_{1}}}=T_{u_{r_{1}}x_{s_{1}}}G_{f_{r_{1}}}$ (note that
$u_{r_{1}}=g_{f_{r_{1}}}$ and $x_{s_{1}}=h_{l_{s_{1}}}$).
Hence the  diagram

\[
\xymatrix{
V_{r_{\resizebox{.28em}{.28em}{1}}}\ar[r]
\ar[d]_{F_{r_{\resizebox{.26em}{.26em}{1}}}} &
V_{s_{\resizebox{.28em}{.28em}{1}}}
\ar[d]^{L_{s_{\resizebox{.26em}{.26em}{1}}}}\\
W_{f_{r_{\resizebox{.26em}{.26em}{1}}}}
\ar[r]\ar[d]_{G_{f_{r_{\resizebox{.26em}{.26em}{1}}}}}
 & W_{l_{s_{\resizebox{.26em}{.26em}{1}}}}
 \ar[d]^{H_{l_{s_{\resizebox{.26em}{.26em}{1}}}}} \\
 Z_{u_{r_{\resizebox{.26em}{.26em}{1}}}}\ar[r]
 &Z_{x_{s_{\resizebox{.26em}{.26em}{1}}}}
 }
\]

\noindent commutes, and then we have $r_{1}<s_{1}$, $u_{r_{1}}<x_{s_{1}}$, and
$X_{s_{1}}E_{r_{1}s_{1}}=T_{u_{r_{1}}x_{s_{1}}}U_{r_{1}}$.

There  is $k_{2}\geq 1$ such that $q_{k_{2}}> l_{s_{1}}$. There is
$k_{3}\geq 1$ with $f_{n_{k_{3}}}\geq p_{k_{2}+1}$; set $n_{k_{3}}=r_{2}$.
Each square in the following diagram commutes:

\[
\xymatrix{
W_{l_{s_{\resizebox{.26em}{.26em}{1}}}}\ar[r]
\ar[d]_{H_{l_{s_{\resizebox{.26em}{.26em}{1}}}}}
 & W_{q_{k_{\resizebox{.26em}{.26em}{2}}}}
 \ar[r]\ar[d]_{H_{q_{k_{\resizebox{.26em}{.26em}{2}}}}}
 & W_{p_{k_{\resizebox{.26em}{.26em}{2}}+1}}
 \ar[r] \ar[d]_{G_{p_{k_{\resizebox{.26em}{.26em}{2}}+1}}}
 & W_{f_{r_{2}}}
 \ar[d]_{G_{f_{r_{\resizebox{.26em}{.26em}{2}}}}}  \\
 Z_{h_{l_{s_{\resizebox{.26em}{.26em}{1}}}}}\ar[r]
 &Z_{h_{q_{k_{\resizebox{.26em}{.26em}{2}}}}}
 \ar[r] &Z_{g_{p_{k_{\resizebox{.26em}{.26em}{2}}+1}}}\ar[r]
 &Z_{g_{f_{r_{\resizebox{.26em}{.26em}{2}}}}}\ .
 }
 \]

\noindent Thus, we have $G_{f_{r_{2}}}S_{l_{s_{1}}f_{r_{2}}}=
T_{x_{s_{1}}u_{r_{2}}}H_{l_{s_{1}}}$. Therefore the diagram

\[
\xymatrix{
V_{s_{\resizebox{.28em}{.28em}{1}}}
\ar[r]\ar[d]_{L_{s_{\resizebox{.26em}{.26em}{1}}}} &
V_{r_{\resizebox{.28em}{.28em}{2}}}
\ar[d]^{F_{r_{\resizebox{.26em}{.26em}{2}}}}\\
W_{f_{r_{\resizebox{.26em}{.26em}{1}}}}\ar[r]
\ar[d]_{H_{l_{s_{\resizebox{.26em}{.26em}{1}}}}}
 & W_{l_{s_{\resizebox{.26em}{.26em}{1}}}}
 \ar[d]^{G_{f_{r_{\resizebox{.26em}{.26em}{2}}}}}\\
 Z_{x_{s_{\resizebox{.26em}{.26em}{1}}}}\ar[r]
 &Z_{u_{r_{\resizebox{.26em}{.26em}{2}}}}
 }
\]

\noindent commutes, and then we have $s_{1}<r_{2}$, $x_{s_{1}}<u_{r_{2}}$, and
$U_{r_{2}}E_{s_{1}r_{2}}=T_{x_{s_{1}}u_{r_{2}}}X_{s_{1}}$.

Continuing this procedure, we obtain sequences $(r_{k})_{k=1}^{\infty}$ and
$(s_{k})_{k=1}^{\infty}$ such that
$r_{1}<s_{1}<r_{2}<s_{2}<\cdots$ and
$u_{r_{1}}<x_{s_{1}}<u_{r_{2}}<x_{s_{2}}<\cdots$ and the commutativity
required in Definition~\ref{defeq} for $gf\sim hl$ is valid. Hence, $[gf]=[hl]$.

Finally, since by Proposition~\ref{propre} $\mathbf{BD}$ with  premorphisms is a category, it follows that composition of morphisms as defined, which we have shown is well~defined, makes $\mathbf{BD}$   a category.
\end{proof}
Let us refer to $\mathbf{BD}$, with  morphisms as defined above, as the  \emph{category of Bratteli diagrams}.

It is tempting to propose an alternative definition for equivalence of two premorphisms (cf.~Definition~\ref{defeq})
as follows: $f\sim g$ if the diagram containing both $f$ and $g$
commutes, in the sense that each triangle and each square in the diagram is commutative (alternatively, any two paths with the same endpoints agree).
However, this relation is not transitive, even if we strengthen the definition
of  a premorphism to consist of only embedding matrices
instead of  multiplicity matrices.
The point is that an embedding matrix is not necessarily injective, as is seen with the following example.
\begin{example}
The following embedding matrices $E_{1}$, $E_{2}$, and $E_{3}$ are such that
$E_{3}E_{1}=E_{3}E_{2}$, but $E_{1}\neq E_{2}$. Let
$V_{1}=(1)$,
$V_{2}=
\left(
 \begin{smallmatrix}
 2\\
 2
 \end{smallmatrix}
 \right)$,
$V_{3}=(6)$,
$E_{1}=\left(
 \begin{smallmatrix}
 2\\
 1
 \end{smallmatrix}
 \right)$,
$E_{2}=\left(
 \begin{smallmatrix}
 0\\
 2
\end{smallmatrix}
\right)$, and
$E_{3}=(1\ 2)$.
Thus, $E_{1},E_{2}: V_{1}\to  V_{2}$ and $E_{3}: V_{2}\to  V_{3}$ are embedding
matrices. We have $E_{3}E_{1}=E_{3}E_{2}=(4)$, but $E_{1}\neq E_{2}$. The only thing
we can say is that there is a unitary $u\in C^{*}(V_{3})$ such that
$h(E_{3})h(E_{1})=(\mathrm{Ad}\,{u})h(E_{3})h(E_{2})$, by Lemma~\ref{lempr}
and Lemma~\ref{lemun}, but
$h(E_{1})\neq h(E_{2})$, since $E_{1}\neq E_{2}$. (See the remark following
Lemma~\ref{lempr} for notation.)
\end{example}
Here are two correct alternative formulations of the definition for equivalence of premorphisms.
We shall use the first one in a number of places later.
\begin{definition}\label{defeq2}
Let $f,g: B\to C$ be premorphisms in $\mathbf{BD}$ such that $B=(V,E)$, $C=(W,S)$,
$f=((F_{n})_{n=1}^{\infty},(f_{n})_{n=1}^{\infty})$, and
$g=((G_{n})_{n=1}^{\infty},(g_{n})_{n=1}^{\infty})$. Let us say that $f$ is \emph{equivalent} to $g$, in the second sense,
 if for each $n\geq 1$ there is an $m\geq f_{n},g_{n}$ such that
 $S_{f_{n}m}F_{n}=S_{g_{n}m}G_{n}$.
\end{definition}
\begin{definition}\label{defeq3}
Let $f,g: B\to C$ be premorphisms in $\mathbf{BD}$ such that $B=(V,E)$, $C=(W,S)$,
$f=((F_{n})_{n=1}^{\infty},(f_{n})_{n=1}^{\infty})$, and
$g=((G_{n})_{n=1}^{\infty},(g_{n})_{n=1}^{\infty})$. Let us say that $f$ is \emph{equivalent} to $g$, in the third sense,
 if for each $n\geq 1$ and for each $k\geq n$, there is an
 $m\geq f_{n},g_{k}$ such that $S_{f_{n}m}F_{n}=S_{g_{k}m}G_{k}E_{nk}$.
\end{definition}
Let us show that these two definitions are equivalent to Definition \ref{defeq}.
\begin{proposition}\label{proeqpre}
Definitions \ref{defeq}, \ref{defeq2}, and \ref{defeq3} are
 equivalent.
\end{proposition}
\begin{proof}
The fact that these definitions are equivalent is based on the following observation.
If we assume that a pair of premorphisms are equivalent in the sense of any one of these definitions, then the union of the
corresponding diagrams is commutative at infinity, in the sense that any two
paths with the same endpoints agree, after going sufficiently further
out, i.e., composing with a long enough subsequent path.  In fact, this is, in each sense, just a reformulation of the definition. (But let us proceed, more prosaically perhaps, in cyclic order.)

\emph{Definition~\ref{defeq} implies Definition~\ref{defeq2}}: Suppose that $f\sim g$ in the sense
of Definition~\ref{defeq}. Let $n\geq 1$ and assume that $f_{n}\leq g_{n}$.
There is $k\geq 1$ such that $n_{k}\geq n$. Thus,
$f_{n}\leq f_{n_{k}}$ and $g_{n} \leq g_{m_{k}}$.
Using Definition~\ref{defeq} and the fact that $f$, $g$ are premorphisms
we have
\begin{align}
S_{f_{n}g_{m_{k}}}F_{n} &=
S_{f_{n_{k}}g_{m_{k}}}S_{f_{n}f_{n_{k}}}F_{n}\notag\\
&= S_{f_{n_{k}}g_{m_{k}}}F_{n_{k}}E_{nn_{k}}\notag\\
&= G_{m_{k}}E_{n_{k}m_{k}}E_{nn_{k}}\notag\\
&=S_{g_{n}g_{m_{k}}}G_{n}. \notag
\end{align}
Therefore, if
$f_{n}\leq g_{n}$ we have
$S_{f_{n}g_{m_{k}}}F_{n}=S_{g_{n}g_{m_{k}}}G_{n}$.
Similarly, if $g_{n}< f_{n}$, there is $l\geq 1$ such that
$g_{n}\leq f_{n}\leq f_{m_{l}}$ and
$S_{f_{n}f_{m_{l}}}F_{n}=S_{g_{n}f_{m_{l}}}G_{n}$.
Set $\max (g_{m_{k}},f_{m_{l}})=m$.
Then we have $S_{f_{n}m}F_{n}=S_{g_{n}m}G_{n}$.

\emph{Definition~\ref{defeq2} implies Definition~\ref{defeq3}}:
Suppose that $k$, $n$ are positive integers with $k\geq n$.
Applying Definition~\ref{defeq2}, we get $m\geq f_{k},g_{k}$
such that $S_{f_{k}m}F_{k}=S_{g_{k}m}G_{k}$.
Using Definition~\ref{defeq2} and the fact that $f$ is a premorphism
we have
\begin{align}
S_{f_{n}m}F_{n} &=
S_{f_{k}m}S_{f_{n}f_{k}}F_{n}\notag\\
&= S_{f_{k}m}F_{k}E_{nk}\notag\\
&= S_{g_{k}m}G_{k}E_{nk}. \notag
\end{align}
Therefore,
$S_{f_{n}m}F_{n}=S_{g_{k}m}G_{k}E_{nk}$.

\emph{Definition~\ref{defeq3} implies Definition~\ref{defeq}}:
Set $n_{1}=1$. Applying Definition~\ref{defeq3}, one obtains
$m\geq f_{n_{1}}, g_{n_{1}}$ such that
$S_{f_{n_{1}}m}F_{n}=S_{g_{n_{1}}m}G_{n}$.
Since $\{g_{k}\mid k\geq 1\}$ is cofinal in $\mathbb{N}$,
there is $m_{1}>n_{1}$ such that $g_{m_{1}}>m$.
Using Definition~\ref{defeq2} and the fact that $g$ is a premorphism
we have
\begin{equation}
S_{f_{n_{1}}g_{m_{1}}}F_{n_{1}} =
S_{g_{n_{1}}g_{m_{1}}}G_{n_{1}}
= G_{m_{1}}E_{n_{1}m_{1}}\notag.
\end{equation}
Continuing this procedure, we obtain sequences
$(n_{k})_{k=1}^{\infty}$ and $(m_{k})_{k=1}^{\infty}$ of
positive integers that satisfy Definition~\ref{defeq}.
\end{proof}
\section{The Category of AF algebras and the Functor $\mathcal{B}$}
\noindent To define the category of AF algebras $\mathbf{AF}$, in such a way that we will be able
to define a functor $\mathcal{B}$ from $\mathbf{AF}$ to $\mathbf{BD}$,
first we need to identify exactly what  the Bratteli diagram of an AF algebra depends on.

Let $A=\overline{\bigcup_{n\geq 1}A_{n}}$ be an AF~algebra, where
$(A_{n})_{n=1}^{\infty}$ is an increasing sequence of finite dimensional
C$^{*}$-subalgebras of $A$. Since there are infinitely  many sequences with this
property, we need to fix one of them. Each $A_{n}$ is $\ast$-isomorphic to a
finite dimensional C$^{*}$-algebra
$A_{n}^{'}=\mathcal{M}_{m_{1}}\oplus\cdots \oplus\mathcal{M}_{m_{l}}$ via
the map $\varphi_{n}: A_{n}\to A_{n}'$. We then obtain the $\ast$-homomorphism
$\varphi_{n}': A_{n}'\to A_{n+1}'$ with $\varphi_{n}'=\varphi_{n+1}\varphi_{n}^{-1}$, and the following diagram commutes:

\[
\xymatrix{
A_{n}\ar@{^{(}->}[r]\ar[d]_{\varphi_{n}} & A_{n+1} \ar[d]^{\varphi_{n+1}}\\
A_{n}'\ar[r]_{\varphi_{n}'} & A_{n+1}'\ \ .
 }
\]

\noindent By Theorem~\ref{thrfd}, there is a  multiplicity  matrix $E_{n}$ corresponding
to $\varphi_{n}'$; that is, $E_{n}=R_{\varphi_{n}'}$. But $E_{n}$ depends on $\varphi_{n}$
and $\varphi_{n+1}$, as a different choice of $\varphi_{n}$ permuting identical direct summands of
$A_{n}'$ of course results in a different $R_{\varphi_{n}'}$.

\begin{definition}\label{defaf}
Let $\mathbf{AF}$ denote the category whose objects are all triples
$(A, (A_{n})_{n=1}^{\infty},(\varphi_{n})_{n=1}^{\infty})$ where $A$
is an AF algebra, $(A_{n})_{n=1}^{\infty}$ is an increasing sequence of
finite dimensional C$^{*}$-subalgebras of $A$ such that
$\bigcup_{n=1}^{\infty}A_{n}$ is dense in $A$, and each $\varphi_{n}$ is a
$*$-isomorphism from $A_{n}$ onto a  C$^{*}$-algebra
$A_{n}^{'}=\mathcal{M}_{m_{1}}\oplus\cdots \oplus\mathcal{M}_{m_{l}}$
for some $m_{1},\ldots,m_{l}$ depending on $A_{n}$.

Let $\mathcal{A}_{1}=(A, (A_{n})_{n=1}^{\infty},(\varphi_{n})_{n=1}^{\infty})$ and
$\mathcal{A}_{2}=(B, (B_{n})_{n=1}^{\infty},(\psi_{n})_{n=1}^{\infty})$
be in $\mathbf{AF}$.
By a  \emph{morphism} $\varphi$ from $\mathcal{A}_{1}$ to
$\mathcal{A}_{2}$ let us just mean a $*$-homomorphism from $A$ to $B$. Then $\mathbf{AF}$ with  morphisms thus defined is a category.
Let us call $\mathbf{AF}$ with  morphisms as defined the \emph{category of AF algebras}.
\end{definition}
\begin{remark}
In Definition~\ref{defaf}, we have fixed  a sequence of $*$-isomorphisms
$(\varphi_{n})_{n=1}^{\infty}$ for the AF~algebra
$(A, (A_{n})_{n=1}^{\infty},(\varphi_{n})_{n=1}^{\infty})$
to  be able to associate a particular Bratteli diagram to the algebra.
In \cite{br72}, Bratteli also fixed
a sequence of systems of matrix units for the AF algebra
to  be able to associate  the diagram. These are equivalent procedures.
\end{remark}
Next we quote a result of Bratteli with slight changes \cite{br72}, which is used
to justify Definition~\ref{deffu}, below, and in a number of places later.
Before that, let us fix the following notation which will be used frequently.
We need this to avoid restricting the results to just the unital case.
\begin{notation}
Throughout this note, for a C$^*$algebra $A$, we shall use two
(minimal) unitizations $A^{\sim}$ and $A^{+}$ as defined in \cite{wo93}. In fact, when $A$ is not unital,
both of them are equal and contain $A$ as
a maximal ideal of codimension one.
When $A$  is unital, $A^{\sim}=A$ but again $A^{+}$ contains $A$ as
a maximal ideal of codimension one. The units of $A^{\sim}$ and $A^{+}$
will both be denoted by 1.
\end{notation}
\begin{lemma}\label{lembr}
Let $A=\overline{\bigcup_{n\geq  1}A_{n}}$ be an AF algebra where
$(A_{n})_{n=1}^{\infty}$ is an increasing sequence of finite dimensional  C$^{*}$-subalgebras
of $A$. Let $B$ be a finite dimensional  C$^{*}$-subalgebra
of $A$. Then for each $\varepsilon >0$ there is a unitary $u\in A^{\sim}$
with $\Vert u-1\Vert<\varepsilon$ and a positive integer $n$ such that
$uBu^{*}\subseteq A_{n}$.
\end{lemma}
\begin{proof}
In the unital case, the lemma is essentially \cite[Lemma~2.3]{br72}. To deal with the non-unital case,
just add a unit to both algebras and use the same lemma.
\end{proof}
\begin{definition}\label{deffu}
Let $\mathcal{A}=(A, (A_{n})_{n=1}^{\infty},(\varphi_{n})_{n=1}^{\infty})$
be in the category  $\mathbf{AF}$, and let us define $\mathcal{B}(\mathcal{A})$ in $\mathbf{BD}$
as follows. Consider the given isomorphisms $\varphi_{n}: A_{n}\to A_{n}'$ and define
$\varphi_{n}': A_{n}'\to A_{n+1}'$ by $\varphi_{n}'=\varphi_{n+1}\varphi_{n}^{-1}$,
for each $n\geq 1$;
then the following diagram commutes, for each $n$:

\[
\xymatrix{
A_{n}\ar@{^{(}->}[r]\ar[d]_{\varphi_{n}} & A_{n+1} \ar[d]^{\varphi_{n+1}}\\
A_{n}'\ar[r]_{\varphi_{n}'} & A_{n+1}'\ .
 }
\]

\noindent Write  $A_{n}^{'}=\mathcal{M}_{m_{n1}}\oplus\cdots \oplus\mathcal{M}_{m_{nk_{n}}}$, and set
\begin{equation*}
V_{n}=\left(
 \begin{smallmatrix}
 m_{n1}\\
 \vdots\\
 m_{nk_{n}}
 \end{smallmatrix}
 \right).
\end{equation*}
\noindent
By Theorem~\ref{thrfd}, there is a unique embedding  matrix $E_{n}$ corresponding
to $\varphi_{n}'$; that is, $E_{n}=R_{\varphi_{n}'}$. Set
$\mathcal{B}(\mathcal{A})=((V_{n})_{n=1}^{\infty}, (E_{n})_{n=1}^{\infty})$.

Let $\varphi: \mathcal{A}_{1}\to \mathcal{A}_{2}$ be a morphism
in $\mathbf{AF}$ where
$\mathcal{A}_{1}=(A, (A_{n})_{n=1}^{\infty},(\varphi_{n})_{n=1}^{\infty})$ and
$\mathcal{A}_{2}=(B, (B_{n})_{n=1}^{\infty},(\psi_{n})_{n=1}^{\infty})$, in other words, a $*$-homomorphism from $A$ to $B$.
Define
$\mathcal{B}(\varphi): \mathcal{B}(\mathcal{A}_{1})\to \mathcal{B}(\mathcal{A}_{2})$
as follows.
There is
an $f_{1}\geq 1$ and  a unitary $u_{1}\in B^{\sim}$ such that
$u_{1}\varphi(A_{1})u_{1}^{*}\subseteq B_{f_{1}}$ and
$\Vert u_{1}-1\Vert <\frac{1}{2}$, by Lemma~\ref{lembr}.
Let $g_{1}: A_{1}\to B_{f_{1}}$ be such that
$g_{1}= (\mathrm{Ad}\,{u_{1}})\circ \varphi\upharpoonright_{A_{1}}$. Now
define  $\eta_{1}: A_{1}'\to B_{f_{1}}'$ by
$\eta_{1}=\psi_{f_{1}}g_{1}\varphi_{1}^{-1}$.
Denote by $F_{1}$
 the   multiplicity matrix corresponding to  $\eta_{1}$,
according to Theorem~\ref{thrfd}; that is, $F_{1}=R_{\eta_{1}}$.

Similarly, choose unitaries
$u_{2},u_{3},\ldots$ in $B^{\sim}$ and positive integers $f_{2},f_{3},\ldots$
with $f_{1}\leq f_{2}\leq\cdots$  such that $u_{n}\varphi(A_{n})u_{n}^{*}\subseteq B_{f_{n}}$,
$\Vert u_{n}-1\Vert <\frac{1}{2}$,
for each $n\geq 1$, and the sequence $ (f_{n})_{n=1}^{\infty}$ is cofinal in $\N$.
(The condition $\Vert u_{n}-1\Vert <\frac{1}{2}$ is important and will be used in the proof of
Proposition~\ref{profun}.)
Let $g_{n}: A_{n}\to B_{f_{n}}$ be such that
$g_{n}= (\mathrm{Ad}\,{u_{n}})\circ \varphi\upharpoonright_{A_{n}}$ and
define  $\eta_{n}: A_{n}'\to B_{f_{n}}'$ by
$\eta_{n}=\psi_{f_{n}}g_{n}\varphi_{n}^{-1}$;
set $R_{\eta_{n}}=F_{n}$.
We will show that $((F_{n})_{n=1}^{\infty}, (f_{n})_{n=1}^{\infty})$ is a premorphism
from $\mathcal{B}(\mathcal{A}_{1})$ to $\mathcal{B}(\mathcal{A}_{2})$
(see Proposition~\ref{profun}).
Denote by  $\mathcal{B}(\varphi)$  the equivalence class of the premorphism
$((F_{n})_{n=1}^{\infty}, (f_{n})_{n=1}^{\infty})$ (as in Definition~\ref{defeq}).
The following (a priori non-commutative) diagram  illustrates  the idea of this definition:

\[
\xymatrix{
A_{1}'\ar[r]^{\varphi_{1}'}\ar[d]_{\varphi_{1}^{-1}}
& A_{2}'\ar[r]^{\varphi_{2}'}\ar[d]_{\varphi_{2}^{-1}}
& A_{3}'\ar[r]^{\varphi_{3}'}\ar[d]_{\varphi_{3}^{-1}}
& \cdots\ \\
A_{1}\ar@{^{(}->}[r]\ar[d]_{(\mathrm{Ad}\,{u_{1}}) \varphi}
& A_{2}\ar@{^{(}->}[r] \ar[rd]_(.3){(\mathrm{Ad}\,{u_{2}}) \varphi}
& A_{3}\ar@{^{(}->}[r] \ar[rd]_(.3){(\mathrm{Ad}\,{u_{3}}) \varphi} & \cdots\ \\
B_{f_{1}}\ar@{^{(}->}[r]\ar[d]_{\psi_{1}}& B_{2}\ar@{^{(}->}[r] \ar[d]_{\psi_{2}}
& B_{f_{2}}\ar@{^{(}->}[r]\ar[d]_{\psi_{3}} & \cdots\ \\
B_{1}'\ar[r]_{\psi_{1}'}
& B_{2}'\ar[r]_{\psi_{2}' }
& B_{3}'\ar[r]_{\psi_{3}' }
& \cdots\ . }
\]
In an obvious way, the diagram associated to the zero algebra
in $\mathbf{AF}$
is the zero Bratteli diagram, and the morphism associated to the zero
homomorphism is the zero morphism.
\end{definition}
We shall need a number of lemmas to show that   $\mathcal{B}: \mathbf{AF} \to \mathbf{BD}$ is
well defined, i.e., that $\mathcal{B}(\varphi)$ is independent of the choice of
sequences  $ (f_{n})_{n=1}^{\infty}$ and  $ (u_{n})_{n=1}^{\infty}$.
The  first two lemmas are well known  in the special case of injective $\ast$-homomorphisms.
Before that, let us introduce an important notion (or perhaps just notation!)  that will be used frequently.
\begin{remark}
Let $\varphi: A\to B$ be a $\ast$-homomorphism  between
finite dimensional C$^*$-algebras. By the  multiplicity matrix of $\varphi$,
 $R_{\varphi}$, we mean that there have been implicitly fixed  two $*$-isomorphisms
 $\varphi_{1} :A\to A'$ and
 $\varphi_{2} :B\to B'$, where $A'$ and $B'$ are finite dimensional C$^*$-algebras
$\mathcal{M}_{n_{1}}\oplus\cdots \oplus\mathcal{M}_{n_{k}}$ and
 $\mathcal{M}_{m_{1}}\oplus\cdots \oplus\mathcal{M}_{m_{l}}$,
 respectively, and $R_{\varphi}$ is the  multiplicity matrix of
 $\varphi_{2}\varphi \varphi_{1}^{-1} : A'\to B'$, according to Theorem~\ref{thrfd}.
 \end{remark}
 \begin{lemma}\label{lempr}
 Let $\varphi: A\to B$ and $\psi : B\to C$ be $\ast$-homomorphisms between finite dimensional C$^{*}$-algebras. Then  $R_{\psi \varphi}=R_{\psi}R_{\varphi}$.
 \end{lemma}
 \begin{proof}
 One can give a proof for the case of injective $\ast$-homomorphisms using
 the matrix units  \cite[Lemma~15.3.2]{lbr92}, and it is easy to conclude it for the general case using the matrix language and Theorem~\ref{thrfd}.
 \end{proof}
 \begin{remark}
 Let $V_{i}$ be a column matrix of non-zero positive integers for $i=1,2$ and let
 $E:V_{1}\to V_{2}$ be a  multiplicity matrix. Write $V_{1}^{T}=(n_{1} \cdots \ n_{k})$ and
 $V_{2}^{T}=(m_{1} \cdots \ m_{l})$. Let $C^{*}(V_{1})$ denote the
 C$^{*}$-algebra
 $\mathcal{M}_{n_{1}}\oplus\cdots \oplus\mathcal{M}_{n_{k}}$ and similarly for $C^{*}(V_{2})$.
Write $E=(a_{ij})$. Then there is a canonical  $\ast$-homomorphism
 $h(E): C^{*}(V_{1})\to C^{*}(V_{2})$ which is defined by

  \begin{equation*}
 \pi_{i}(h(E)(u_{1},\ldots, u_{k}))=
 \left(
\begin{smallmatrix}
  \resizebox{.9cm}{.4cm}{$u_{1}^{(a_{i1})}$} &  & & \resizebox{.2cm}{.22cm}{$0$}\ \ \ \ \\
  & \ddots &  & \\
  &  & \resizebox{.9cm}{.4cm}{$u_{k}^{(a_{ik})}$} &  \\
 \resizebox{.2cm}{.22cm}{$0$} & & & \resizebox{.67cm}{.31cm}{$0^{(s_{i})}$}
 \end{smallmatrix}
 \right) \
  \text{for}\  (u_{1},\ldots, u_{k})\in C^{*}(V_{1}),
 \end{equation*}

 \noindent where $s_{i}$ is such that
 $\sum_{j=1}^{k}a_{ij}n_{j}+ s_{i}= m_{i}$, for each $1\leq
i\leq l$.
 Recall that for a unital C$^{*}$-algebra $A$ and a unitary element $u$ in $A$, the
 $\ast$-isomorphism $\mathrm{Ad}\,{u}:A\to A$ is defined by
 $(\mathrm{Ad}\,{u})(a)=uau^{*}$  ($a\in A$).
 In particular, the conclusion of Theorem~\ref{thrfd}
 could be summarized as $\varphi=(\mathrm{Ad}\,{u}) h(E)$, where $u=v^{*}$.
 \end{remark}
We shall need the following lemma in a number of places later.
It is proved in \cite[Theorem I.11.9]{ta79}, and also is obvious in view of
 Theorem~\ref{thrfd} and the remark above.
 \begin{lemma}\label{lemun}
 Let $\varphi , \psi: A\to B$  be $\ast$-homomorphisms between finite dimensional C$^{*}$-algebras.
 Then $R_{\varphi}=R_{\psi}$ if and only if there is a unitary $u$ in $B$
 such that  $\varphi= (\mathrm{Ad}\,{u})\psi$.
\end{lemma}
The following corollary (used, if not explicitly stated, in \cite{br72}) is given in the case of
injective $*$-homomorphisms in \cite[Lemma~15.3.2]{lbr92}.
\begin{corollary}\label{corcom}
Let $V_{i}$ be a column matrix of non-zero positive integers for $1\leq i\leq 4$. Let $E_{1}:V_{1}\to V_{2}$,
$E_{2}: V_{3}\to V_{4}$, $E_{3}:V_{1}\to V_{3}$, and $E_{4}: V_{2}\to V_{4}$ be
multiplicity matrices such that the diagram

\[
\xymatrix{
V_{1}\ar[r]^{E_{1}}\ar[d]_{E_{3}} & V_{2} \ar[d]^{E_{4}}\\
V_{3}\ar[r]_{E_{2}} & V_{4}
 }
\]

\noindent commutes; that is,
$E_{4}E_{1}=E_{2}E_{3}$.
Let $\varphi_{1}:C^{*}(V_{1})\to C^{*}(V_{2})$,
$\varphi_{2}: C^{*}(V_{3})\to C^{*}(V_{4})$, $\varphi_{3}:C^{*}(V_{1})\to C^{*}(V_{3})$, and $\varphi_{4}: C^{*}(V_{2})\to C^{*}(V_{4})$ be $*$-homomorphisms such that $R_{\varphi_{i}}=E_{i}$ for
$1\leq i\leq 4$. Then there is a unitary $u\in C^{*}(V_{4})$ such that the following
diagram commutes:

\[
\xymatrix{
C^{*}(V_{1})\ar[r]^{\varphi_{1}}\ar[d]_{\varphi_{3}}
& C^{*}(V_{2}) \ar[d]^{(\mathrm{Ad}\,{u})\varphi_{4}}\\
C^{*}(V_{3})\ar[r]_{\varphi_{2}} & C^{*}(V_{4})\ ;
 }
\]
i.e., $(\mathrm{Ad}\,{u})\varphi_{4}\varphi_{1}=\varphi_{2}\varphi_{3}$.
\end{corollary}
\begin{proof}
By Lemma~\ref{lempr} we have
$R_{\varphi_{4}\varphi_{1}}=E_{4}E_{1}=E_{2}E_{3}=R_{\varphi_{2}\varphi_{3}}$.
Thus by Lemma~\ref{lemun}, there is a unitary $u\in C^{*}(V_{4})$ such that
$(\mathrm{Ad}\,{u})\varphi_{4}\varphi_{1}=\varphi_{2}\varphi_{3}$.
\end{proof}
Next we give a slight modification of \cite[Lemma~2.4]{br72}.
Since our $*$-homomorphisms are not assumed to be unital, we prove
a non-unital version
which is suitable for our purposes.
This lemma  gives a criterion to check whether
two $*$-homomorphisms between finite dimensional
C$^{*}$-algebras have the same  multiplicity matrices.
We will use it to show that the functor
$\mathcal{B}: \mathbf{AF} \to \mathbf{BD}$ is well defined
(in the proof of Proposition~\ref{profun}) and in a number of places later.
\begin{lemma}\label{lempart}
Let $\varphi,\psi: A\to B$ be  $*$-homomorphisms between finite dimensional
C$^{*}$-algebras such that $\Vert \varphi - \psi \Vert <1$. Then there is a
unitary $u$ in $B$ such that $\varphi=(\mathrm{Ad}\,{u})\psi$ and hence
$R_{\varphi}=R_{\psi}$.
\end{lemma}
\begin{proof}
Let $\{e_{ij}^{l} : 1\leq l\leq k, \ 1\leq i,j\leq n_{l}\}$ be a set of matrix units for $A$. Set
$p_{ij}^{l}=\varphi(e_{ij}^{l})$ and $q_{ij}^{l}=\psi(e_{ij}^{l})$, for each
$1\leq l\leq k$ and $1\leq i,j\leq n_{l}$.
Fix $1\leq l\leq k$. Since $\Vert p_{11}^{l} - q_{11}^{l} \Vert <1$, by \cite[Lemma~1.8]{gl60}
there is a partial isometry
$w_{l}\in B$ such that $p_{11}^{l}=w_{l}w_{l}^{*}$ and $q_{11}^{l}=w_{l}^{*}w_{l}$. Then
$w_{l}q_{11}^{l}w_{l}^{*}=p_{11}$. Set
\[
\sum_{l=1}^{k}\sum_{i=1}^{n_{l}}p_{i1}^{l}w_{l}q_{1i}^{l}=w.
\]
Thus we have $wq_{ij}^{l}w^{*}=p_{ij}^{l}$, for the above values of $i,j,l$. Also we have $ww^{*}=\varphi(1)$ and $w^{*}w=\psi(1)$, and so $w$ is a partial isometry from $\psi(1)$ to $\varphi(1)$.
If $\varphi$ and $\psi$ are unital, the proof is complete at this point.
Since $B$ is finite dimensional, there is a partial isometry $v\in B$ such that $w+v$ is  unitary and
$wv^{*}=w^{*}v=0$; set $w+v=u$. We have
$vq_{ij}^{l}=v\psi(1)q_{ij}^{l}= vw^{*}wq_{ij}^{l}=0$, and similarly $q_{ij}^{l}v^{*}=0$,
for the above values of $i,j,l$. Therefore, $uq_{ij}^{l}u^{*}=wq_{ij}^{l}w^{*}=p_{ij}^{l}$,
and so $\varphi=(\mathrm{Ad}\,{u})\psi$.
\end{proof}
\begin{proposition}\label{profun}
$\mathcal{B}: \mathbf{AF} \to \mathbf{BD}$ is a functor.
\end{proposition}
\begin{proof}
First let us  show that $\mathcal{B}$ is well defined. Following the notation of
Definition~\ref{deffu}, we need to show first that
$((F_{n})_{n=1}^{\infty}, (f_{n})_{n=1}^{\infty})$
is a premorphism from $\mathcal{B}(\mathcal{A}_{1})=(V,E)$ to
$\mathcal{B}(\mathcal{A}_{2})=(W,S)$. Fix $n\geq 1$. Consider the following
(a priori non-commutative) diagram:

\[
\xymatrix{
A_{n}\ar@{^{(}->}[r]\ar[d]_{g_{n}} & A_{n+1} \ar[d]^{g_{n+1}}\\
B_{f_{n}}\ar@{^{(}->}[r]
& B_{f_{n+1}}\ .
}
\]

\noindent
Since
$\Vert u_{n}-1\Vert <\frac{1}{2}$,    we have
$\Vert g_{n} - g_{n+1}\Vert_{A_{n}}
\leq\Vert (\mathrm{Ad}\,{u_{n}}) \varphi - (\mathrm{Ad}\,{u_{n+1}}) \varphi\Vert
\leq 2\Vert u_{n}-u_{n+1}\Vert<1$, for each $n\geq 1$.
Applying Lemma~\ref{lempart} for
$g_{n}:A_{n}\to B_{f_{n+1}}$ and
$g_{n+1}\upharpoonright_{A_{n}}:A_{n}\to B_{f_{n+1}}$,
 we get
$
F_{n+1}E_{n}=S_{f_{n}f_{n+1}}F_{n}
$.
This shows that $((F_{n})_{n=1}^{\infty}, (f_{n})_{n=1}^{\infty})$
is a premorphism.

To check that the morphism $\mathcal{B}(\varphi)$ is well defined
we also need to show that it is independent of the choice of $u_{n}$'s and
$f_{n}$'s. Therefore,
 let $(v_{n})_{n=1}^{\infty}$ be
another sequence of unitaries  in $B^{\sim}$ and
$(h_{n})_{n=1}^{\infty}$  an increasing cofinal sequence of positive integers
such that
$v_{n}\varphi(A_{n})v_{n}^{*}\subseteq B_{h_{n}}$ and $\Vert v_{n}-1\Vert <\frac{1}{2}$,
for each $n\geq 1$. Let $k_{n}: A_{n}\to B_{h_{n}}$ be such that
$k_{n}= (\mathrm{Ad}\,{v_{n}})\circ \varphi\upharpoonright_{A_{n}}$ and set $H_{n}=R_{k_{n}}$, the
multiplicity matrix of $k_{n}$.
We have to show that the two premorphisms
$((F_{n})_{n=1}^{\infty},(f_{n})_{n=1}^{\infty})$ and
$((H_{n})_{n=1}^{\infty},(h_{n})_{n=1}^{\infty})$ are equivalent.

Fix $n\geq 1$. We may assume, without loss of generality, that $f_{n}\leq h_{n}$. Then we
have the following (a priori non-commutative) diagram:

\[
\xymatrix{A_{n}\ar[d]_{g_{n}}\ar[rd]^{k_{n}}
 &  \\
 B_{f_{n}}\ar@{^{(}->}[r]
 & B_{h_{n}}\ .
 }
\]

\noindent We have
$\Vert g_{n} - k_{n}\Vert
\leq\Vert (\mathrm{Ad}\,{u_{n}})\varphi - (\mathrm{Ad}\,{v_{n}}) \varphi\Vert
\leq 2\Vert u_{n}-v_{n}\Vert<1$.
Applying Lemma~\ref{lempart} for
$g_{n}:A_{n}\to B_{h_{n}}$ and
$h_{n}:A_{n}\to B_{h_{n}}$,
 we conclude that the following diagram is commutative:

\[
\xymatrix{V_{n}\ar[d]_{F_{n}}\ar[rd]^{H_{n}}
 &  \\
 W_{f_{n}}\ar[r]_{S_{f_{n}h_{n}}}
 &W_{h_{n}}\ .
 }
\]

\noindent Therefore, by Proposition~\ref{proeqpre}, the premorphisms $((F_{n})_{n=1}^{\infty},(f_{n})_{n=1}^{\infty})$ and
$((H_{n})_{n=1}^{\infty},(h_{n})_{n=1}^{\infty})$ are equivalent.
This completes the proof that the map $\mathcal{B}$ is well defined.

Suppose that $\varphi: \mathcal{A}_{1}\to \mathcal{A}_{2}$ and
$\psi:\mathcal{A}_{2}\to \mathcal{A}_{3}$ are morphisms in $\mathbf{AF}$. Let us show
that $\mathcal{B}(\psi\varphi)=\mathcal{B}(\psi)\mathcal{B}(\varphi)$.
Write
$\mathcal{A}_{1}=(A, (A_{n})_{n=1}^{\infty},(\varphi_{n})_{n=1}^{\infty})$,
$\mathcal{A}_{2}=(B, (B_{n})_{n=1}^{\infty},(\psi_{n})_{n=1}^{\infty})$,
$\mathcal{A}_{3}=(C, (C_{n})_{n=1}^{\infty},(\eta_{n})_{n=1}^{\infty})$,
$\mathcal{B}(\mathcal{A}_{1})=(V,E)$,
$\mathcal{B}(\mathcal{A}_{2})=(W,S)$, and
$\mathcal{B}(\mathcal{A}_{3})=(Z,T)$. Thus, $\varphi$ and $\psi$ are $*$-homomorphisms from $A$ to $B$ and from $B$ to $C$, respectively.
Choose a sequence of unitaries
$(u_{n})_{n=1}^{\infty}$ in $B^{\sim}$ with
$\Vert u_{n}-1\Vert <\frac{1}{4}$, $n\in \N$,
and a sequence of positive integers
$(f_{n})_{n=1}^{\infty}$ which construct
a  premorphism
$((F_{n})_{n=1}^{\infty}, (f_{n})_{n=1}^{\infty})$ for
$\mathcal{B}(\varphi)$, according to Definition~\ref{deffu}. Similarly,
choose
$(v_{n})_{n=1}^{\infty}$ with
$\Vert v_{n}-1\Vert <\frac{1}{8}$, $n\in \N$,
and $(g_{n})_{n=1}^{\infty}$ which give
a premorphism
$((G_{n})_{n=1}^{\infty}, (g_{n})_{n=1}^{\infty})$ for
$\mathcal{B}(\psi)$
and choose
$(w_{n})_{n=1}^{\infty}$
with
$\Vert w_{n}-1\Vert <\frac{1}{8}$, $n\in \N$,
and $(h_{n})_{n=1}^{\infty}$ which give
a premorphism
$((H_{n})_{n=1}^{\infty}, (h_{n})_{n=1}^{\infty})$ for
$\mathcal{B}(\psi\varphi)$, according to Definition~\ref{deffu}.

Fix a positive integer $n$. Then $(\mathrm{Ad}\,{u_{n}})\varphi(A_{n})\subseteq B_{f_{n}}$,
$(\mathrm{Ad}\,{v_{f_{n}}})\psi(B_{f_{n}})\subseteq C_{g_{f_{n}}}$, and
$(\mathrm{Ad}\,{w_{n}})\psi\varphi(A_{n})\subseteq C_{h_{n}}$.
We may  assume,
without loss of generality, that $g_{f_{n}}\leq h_{n}$. Then we have the following
(a priori non-commutative) diagram:

\[
\xymatrix{A_{n}\ar[d]_{(\mathrm{Ad}\,{u_{n}})\varphi}\ar[rdd]^{(\mathrm{Ad}\,{w_{n}})\psi\varphi}
 &  \\
 B_{f_{n}}\ar[d]_{(\mathrm{Ad}\,{v_{f_{n}}})\psi}
 & \\
 C_{g_{f_{n}}} \ar@{^{(}->}[r]
 & C_{h_{n}}\ .
 }
\]

\noindent Let us estimate the distance between the  $*$-homomorphisms
$(\mathrm{Ad}\,{v_{f_{n}}})\psi\circ(\mathrm{Ad}\,{u_{n}})\varphi: A_{n}\to C_{h_{n}}$ and
$(\mathrm{Ad}\,{w_{n}})\psi\varphi : A_{n}\to C_{h_{n}}$. For any $x\in A$ we have:

\begin{align*}
\Vert (\mathrm{Ad}\,{v_{f_{n}}})\psi\circ(\mathrm{Ad}\,{u_{n}})\varphi(x)
- (\mathrm{Ad}\,{w_{n}})\psi\varphi(x) \Vert &\leq
\Vert (\mathrm{Ad}\,{v_{f_{n}}})\psi\circ(\mathrm{Ad}\,{u_{n}})\varphi(x)\\
&-
 (\mathrm{Ad}\,{v_{f_{n}}})\psi\varphi(x) \Vert\\
  &+
\Vert (\mathrm{Ad}\,{v_{f_{n}}})\psi\varphi(x)\\
&- (\mathrm{Ad}\,{w_{n}})\psi\varphi(x)\Vert\\
&\leq \Vert (\mathrm{Ad}\,{u_{n}})\varphi(x)-\varphi(x)\Vert\\
&+  2\Vert v_{f_{n}}-w_{n}\Vert \Vert x \Vert\\
&\leq 2\Vert u_{n}-1\Vert \Vert x \Vert
+ 2\Vert v_{f_{n}}-w_{n}\Vert \Vert x \Vert.
\end{align*}
Thus,
\[\Vert (\mathrm{Ad}\,{v_{f_{n}}})\psi\circ(\mathrm{Ad}\,{u_{n}})\varphi -
(\mathrm{Ad}\,{w_{n}})\psi\varphi\Vert
\leq 2 \Vert u_{n}-1\Vert
+ 2\Vert v_{f_{n}}-w_{n}\Vert<\frac{1}{2}+\frac{1}{4}+\frac{1}{4}=1. \]

\noindent By Lemma~\ref{lempart},
the homomorphisms
$(\mathrm{Ad}\,{v_{f_{n}}})\psi\circ(\mathrm{Ad}\,{u_{n}})\varphi$
and
$(\mathrm{Ad}\,{w_{n}})\psi\varphi$ have the same multiplicity matrices. Thus, the diagram

\[
\xymatrix{V_{n}\ar[d]_{F_{n}}\ar[rdd]^{H_{n}}
 &  \\
 W_{f_{n}}\ar[d]_{G_{f_{n}}}
 & \\
 Z_{g_{f_{n}}} \ar[r]
 & Z_{h_{n}}
 }
\]

\noindent is commutative. By Definition~\ref{defeq},
the premorphism
$((H_{n})_{n=1}^{\infty}, (h_{n})_{n=1}^{\infty})$ is equivalent to
the composition of
$((G_{n})_{n=1}^{\infty}, (g_{n})_{n=1}^{\infty})$ and
$((F_{n})_{n=1}^{\infty}, (f_{n})_{n=1}^{\infty})$. Note that this composition of premorphisms is an admissible premorphism for $\psi\varphi$ in the sense required in Definition~\ref{deffu}, since
$\Vert w_{f_{n}}u_{n} - 1\Vert <\frac{1}{4}+\frac{1}{8}<\frac{1}{2}$. By Definition~\ref{deffu} (which is vindicated by the fact, proved above, that $\mathcal{B}(\psi\varphi)$ defined in this way is well defined), this composition represents $\mathcal{B}(\psi\varphi)$, and so
 $\mathcal{B}(\psi\varphi)=\mathcal{B}(\psi)\mathcal{B}(\varphi)$.

The remaining condition,
$\mathcal{B}(\mathrm{id} _{\mathcal{A}})=\mathrm{id}_{\mathcal{B}(\mathcal{A})}$,
is clear.
\end{proof}
The following lemma is used to prove Theorem~\ref{thrst}. The hypothesis of this lemma is just Bratteli's notion of
equivalence of Bratteli diagrams in his paper \cite{br72}---which is easily seen to be the same as isomorphism in our category.
We give a proof for the convenience of the reader.
\begin{lemma}\label{lemeq}
Let $\mathcal{A}_{1}$ and $\mathcal{A}_{2}$ be in $\mathbf{AF}$ with
$\mathcal{A}_{1}=(A, (A_{n})_{n=1}^{\infty},(\varphi_{n})_{n=1}^{\infty})$,
$\mathcal{A}_{2}=(B, (B_{n})_{n=1}^{\infty},(\psi_{n})_{n=1}^{\infty})$,
$\mathcal{B}(\mathcal{A}_{1})=((V_{n})_{n=1}^{\infty}, (E_{n})_{n=1}^{\infty})$, and
$\mathcal{B}(\mathcal{A}_{2})=((W_{n})_{n=1}^{\infty}, (S_{n})_{n=1}^{\infty}))$.
Then
$\mathcal{A}_{1}\cong \mathcal{A}_{2}$ in $\mathbf{AF}$
if and only if
there are sequences $(r_{k})_{k=1}^{\infty}$ and $(t_{k})_{k=1}^{\infty}$ of
positive integers with
$r_{1}<t_{1}<r_{2}<t_{2}<\cdots$, and there are  multiplicity matrices
$R_{k}: V_{r_{k}}\to W_{t_{k}}$ and $T_{k}: W_{t_{k}}\to V_{r_{k+1}}$, for each $k\geq 1$,
such that the following diagram commutes:
\end{lemma}

\[
\xymatrix{
V_{r_{\resizebox{.28em}{.28em}{1}}}
\ar[rr]^{E_{r_{\resizebox{.26em}{.26em}{1}}r_{\resizebox{.26em}{.26em}{2}}}}
\ar[rd]_{R_{1}} &
 & V_{r_{\resizebox{.28em}{.28em}{2}}}
 \ar[rr]^{E_{r_{\resizebox{.26em}{.26em}{2}}r_{\resizebox{.26em}{.26em}{3}}}}
 \ar[rd]_{R_{2}} &
 & V_{r_{\resizebox{.28em}{.28em}{3}}}\ar[rr] & &
  \cdots  \\
& W_{t_{\resizebox{.26em}{.26em}{1}}}
\ar[rr]_{S_{t_{\resizebox{.26em}{.26em}{1}}t_{\resizebox{.26em}{.26em}{2}}}}
\ar[ru]_{T_{1}} &
 &  W_{t_{\resizebox{.26em}{.26em}{2}}}
 \ar[rr]_{S_{t_{\resizebox{.26em}{.26em}{2}}t_{\resizebox{.26em}{.26em}{3}}}}\ar[ru]_{T_{2}} &
 & \cdots \ .&
 }
 \]
\begin{proof}
First suppose that $\mathcal{A}_{1}\cong \mathcal{A}_{2}$. There is a
$*$-isomorphism $\varphi: A\to B$ such that
$\varphi(\bigcup_{n=1}^{\infty}A_{n})\subseteq \bigcup_{n=1}^{\infty}B_{n}$,
by \cite[Lemma~2.6]{br72}.
Set $r_{1}=1$. Since
$\varphi(\bigcup_{n=1}^{\infty}A_{n})\subseteq \bigcup_{n=1}^{\infty}B_{n}$, there is $t_{1}>1$
with $\varphi(A_{r_{1}})\subseteq B_{t_{1}}$. Similarly there is $r_{2}>t_{1}$ such that
$\varphi^{-1}(B_{t_{1}})\subseteq A_{r_{2}}$. Continuing this procedure, we obtain
sequences $(r_{k})_{k=1}^{\infty}$ and $(t_{k})_{k=1}^{\infty}$ with
$r_{1}<t_{1}<r_{2}<t_{2}<\cdots$ such that
$\varphi(A_{r_{k}})\subseteq B_{t_{k}}$ and
$\varphi^{-1}(B_{t_{k}})\subseteq A_{r_{k+1}}$,
for each $k\geq 1$. Note that $C^{*}(V_{n})=A_{n}'$ and $C^{*}(W_{n})=B_{n}'$,
for $n\geq 1$. Fix $k\geq 1$. Define $\varepsilon_{k}: A_{r_{k}}'\to B_{t_{k}}'$ and
$\delta_{k}: B_{t_{k}}'\to A_{r_{k+1}}'$ with
$\varepsilon_{k}(x)=\psi_{t_{k}}(\varphi(\varphi_{r_{k}}^{-1}(x)))$, $x\in A_{r_{k}}'$,
and $\delta_{k}(x)=\varphi_{r_{k+1}}(\varphi^{-1}(\psi_{t_{k}}^{-1}(x)))$, $x\in B_{t_{k}}'$.
Set
\begin{equation*}
\alpha_{k}=\varphi_{r_{k+1}-1}'\,\varphi_{r_{k+1}-2}'\,\cdots \varphi_{r_{k}}'\ \ \
\text{and}\ \ \
\beta_{k}=\psi_{t_{k+1}-1}'\,\psi_{t_{k+1}-2}'\,\cdots \psi_{t_{k}}'.
\end{equation*}
Then $\delta_{k}\varepsilon_{k}=\alpha_{k}$ and $\varepsilon_{k+1}\delta_{k}=\beta_{k}$.
Set $R_{k}=R_{\varepsilon_{k}}$ and $T_{k}=R_{\delta_{k}}$.
By Lemma~\ref{lempr} we have
$T_{k}R_{k}=E_{r_{k}r_{k+1}}$ and $R_{k+1}T_{k}=S_{t_{k}t_{k+1}}$, i.e.,
the above diagram commutes.

Now let us prove the converse. Define $\alpha_{k}$ and $\beta_{k}$ as above. Let
$\varepsilon_{1}=h(R_{1})$ and $\delta_{1}'=h(T_{1})$ (see the remark following
Lemma~\ref{lempr} for the notation $h(\cdot)$). By Corollary~\ref{corcom},
there is a unitary $u\in A_{r_{2}}'$ such that
$(\mathrm{Ad}\,{u})\delta_{1}'\varepsilon_{1}=\alpha_{1}$;
set $\delta_{1}=(\mathrm{Ad}\,{u})\delta_{1}'$, thus $\delta_{1}\varepsilon_{1}=\alpha_{1}$.
Set $\varepsilon_{2}'=h(R_{2})$. Again by Corollary~\ref{corcom},
there is a unitary $v\in B_{t_{2}}'$ such that
$(\mathrm{Ad}\,{v})\varepsilon_{2}'\delta_{1}=\beta_{1}$; set
$\varepsilon_{2}=(\mathrm{Ad}\,{v})\varepsilon_{2}'$, thus
$\varepsilon_{2}\delta_{1}=\beta_{1}$. Continuing this procedure, we
obtain injective $*$-homomorphisms $\varepsilon_{k}: A_{r_{k}}'\to B_{t_{k}}'$ and
$\delta_{k}: B_{t_{k}}'\to A_{r_{k+1}}'$, for each $k\geq 1$, such that the
following diagram commutes:

\[
\xymatrix{
A_{r_{\resizebox{.28em}{.28em}{1}}}'
\ar[rr]^{\alpha_{1}}\ar[rd]_{\varepsilon_{1}} &
 & A_{r_{\resizebox{.28em}{.28em}{2}}}'\ar[rr]^{\alpha_{2}}\ar[rd]_{\varepsilon_{2}} &
 & A_{r_{\resizebox{.28em}{.28em}{3}}}'\ar[rr] & &
  \cdots  \\
& B_{t_{\resizebox{.28em}{.28em}{1}}}'
\ar[rr]_{\beta_{1}}\ar[ru]_{\delta_{1}} &
 &  B_{t_{\resizebox{.28em}{.28em}{2}}}'
 \ar[rr]_{\beta_{2}}\ar[ru]_{\delta_{2}} &
 & \cdots \ .&
 }
 \]

\noindent Let $A'=\varinjlim (A_{r_{k}}', \alpha_{k})$ and
$B'=\varinjlim (B_{t_{k}}', \beta_{k})$. Thus there is an injective $*$-homomorphism
$\varepsilon: A'\to B'$. Let $\alpha^{k}: A_{r_{k}}'\to A'$ and
$\beta^{k}: B_{t_{k}}'\to B'$ be the $*$-homomorphisms that come
from the construction of the direct limit; thus
$\alpha^{k+1}\alpha_{k}=\alpha^{k}$, $\beta^{k+1}\beta_{k}=\beta^{k}$, and
$\varepsilon \alpha^{k}=\beta^{k}\varepsilon_{k}$, for each $k\geq 1$.
We have $\beta^{k}=\beta^{k+1}\beta_{k}=
\beta^{k+1}\varepsilon_{k+1}\delta_{k}=\varepsilon \alpha^{k+1}\delta_{k}$, hence
$\beta^{k}(B_{t_{k}}')\subseteq \varepsilon(A')$, for each $k\geq 1$; thus
$\varepsilon$ is also onto and hence is a $*$-isomorphism. Moreover,
\begin{equation*}
A\cong \varinjlim (A_{k}', \alpha_{k})\cong A' \cong B' \cong
\varinjlim (B_{k}', \alpha_{k})\cong B.
\end{equation*}
Thus there is a $*$-isomorphism
$\varphi: A\to B$ such that
$\varphi (\bigcup_{n=1}^{\infty}A_{n})\subseteq \bigcup_{n=1}^{\infty}B_{n}$.
Therefore,
$\mathcal{A}_{1}\cong \mathcal{A}_{2}$ in $\mathbf{AF}$.
\end{proof}
The following theorem is due to  Bratteli (\cite{br72}, \cite[Proposition III.2.7]{da96}).
\begin{theorem}[Bratteli]\label{olab}
If $A=\overline{\bigcup_{n\geq 1}A_{n}}$ and $B=\overline{\bigcup_{n\geq
1}B_{n}}$ have the same Bratteli diagrams, then they are isomorphic.
\end{theorem}
In fact, as indicated in Lemma~\ref{lemeq} above, Bratteli proved more.
In the setting of Theorem~\ref{olab}, he showed that if the Bratteli diagram of $A$ is equivalent,
in his sense---which is exactly the same as being isomorphic, in our sense, i.e.,
in the category of Bratteli diagrams of Theorem 2.7---, to
the Bratteli diagram of $B$, then $A$ is isomorphic to $B$.

Recall that a functor $F:\mathcal{C}\to \mathcal{D}$
was called in \cite{el10} a \emph{classification functor} if
$F(a)\cong F(b)$ implies $a\cong b$, for each $a,b\in \mathcal{C}$, and a \emph{strong classification functor} if each isomorphism from $F(a)$ onto $F(b)$ is the image of an isomorphism from $a$ to $b$.
With these concepts, one has a functorial formulation of  Bratteli's theorem.
\begin{theorem}\label{thrst}
The functor
 $\mathcal{B}: \mathbf{AF} \to \mathbf{BD}$ is a strong classification functor.
\end{theorem}
\begin{proof}
Let
$\mathcal{A}_{1}=(A, (A_{n})_{n=1}^{\infty},(\varphi_{n})_{n=1}^{\infty})$ and
$\mathcal{A}_{2}=(B, (B_{n})_{n=1}^{\infty},(\psi_{n})_{n=1}^{\infty})$
 in $\mathbf{AF}$ be such that $\mathcal{B}(\mathcal{A}_{1})\cong \mathcal{B}(\mathcal{A}_{2})$.
Write $\mathcal{B}(\mathcal{A}_{1})=((V_{n})_{n=1}^{\infty}, (E_{n})_{n=1}^{\infty})$ and
$\mathcal{B}(\mathcal{A}_{2})=((W_{n})_{n=1}^{\infty}, (S_{n})_{n=1}^{\infty}))$.
There are premorphisms
$f : \mathcal{B}(\mathcal{A}_{1})\to \mathcal{B}(\mathcal{A}_{2})$ and
$g : \mathcal{B}(\mathcal{A}_{2})\to \mathcal{B}(\mathcal{A}_{1})$ such that
$[gf]=[\mathrm{id}_{\mathcal{B}(\mathcal{A}_{1})}]$ and
$[fg]=[\mathrm{id}_{\mathcal{B}(\mathcal{A}_{2})}]$.

Suppose that $f=((F_{n})_{n=1}^{\infty}), (f_{n})_{n=1}^{\infty})$,
$g=((G_{n})_{n=1}^{\infty}), (g_{n})_{n=1}^{\infty})$,
$h=gf$, and $l=fg$ where $h=((H_{n})_{n=1}^{\infty}), (h_{n})_{n=1}^{\infty})$ and
$l=((L_{n})_{n=1}^{\infty}), (l_{n})_{n=1}^{\infty})$. Choose  sequences
$(n_{k})_{k=1}^{\infty}$ and $(m_{k})_{k=1}^{\infty}$ for
$gf\sim \mathrm{id}_{\mathcal{B}(\mathcal{A}_{1})}$
and choose $(p_{k})_{k=1}^{\infty}$ and $(q_{k})_{k=1}^{\infty}$
for $fg\sim \mathrm{id}_{\mathcal{B}(\mathcal{A}_{2})}$, 	
according to Definition~\ref{defeq}. Thus we have
\begin{equation}\label{for1}
E_{n_{k}m_{j}}=E_{h_{n_{k}}m_{j}}H_{n_{k}},\ \  \ \
H_{n_{j}}E_{m_{k}n_{j}}=E_{m_{k}h_{n_{j}}},
\end{equation}
\begin{equation}\label{for2}
S_{p_{k}q_{j}}=S_{l_{p_{k}}{q_{j}}}L_{p_{k}},\ \  \ \
L_{p_{j}}S_{q_{k}p_{j}}=S_{q_{k}l_{p_{j}}},
\end{equation}
for any positive
integers $k$ and $j$ with $k\leq j$. We construct sequences
$(r_{k})_{k=1}^{\infty}$ and $(t_{k})_{k=1}^{\infty}$ to apply
Lemma~\ref{lemeq}. Set $r_{1}=m_{1}$.  The diagram

\[
\xymatrix{
V_{m_{\resizebox{.28em}{.28em}{1}}}
\ar[r]^{E_{m_{\resizebox{.26em}{.26em}{1}}n_{\resizebox{.26em}{.26em}{2}}}}
\ar[dd]_{I}
& V_{n_{\resizebox{.28em}{.28em}{2}}}
\ar[d]^{F_{n_{\resizebox{.26em}{.26em}{2}}}}\\
 & W_{f_{n_{\resizebox{.26em}{.26em}{2}}}}
 \ar[d]^{G_{f_{n_{\resizebox{.26em}{.26em}{2}}}}}\\
 V_{m_{\resizebox{.28em}{.28em}{1}}}
 \ar[r]_{E_{m_{\resizebox{.26em}{.26em}{1}}h_{n_{\resizebox{.26em}{.26em}{2}}}}}
 & V_{h_{n_{\resizebox{.26em}{.26em}{2}}}}
}
\]

\noindent
commutes by Equation~\eqref{for1}, where $I$ is the identity matrix with
suitable size. Set $R_{1}'=F_{n_{2}}E_{m_{1}n_{2}}$. Thus the diagram

\[
\xymatrix{
V_{r_{\resizebox{.28em}{.28em}{1}}}
\ar[rr]^{E_{r_{\resizebox{.26em}{.26em}{1}}h_{n_{\resizebox{.26em}{.26em}{2}}}}}
\ar[rd]_{R_{\resizebox{.26em}{.26em}{1}}'}& &
V_{h_{n_{\resizebox{.26em}{.26em}{2}}}}\\
&W_{f_{n_{\resizebox{.26em}{.26em}{2}}}}
\ar[ru]_{G_{f_{n_{\resizebox{.26em}{.26em}{2}}}}}&
}
\]

\noindent commutes. There is $j\geq 1$ such that $q_{j}>f_{n_{2}}$.
The diagram

\[
\xymatrix{
W_{q_{j}}\ar[r]^{S_{q_{j}p_{j+1}}}\ar[dd]_{I}
& W_{p_{j+1}}\ar[d]^{G_{p_{j+1}}}\\
 & V_{g_{p_{j+1}}}\ar[d]^{F_{g_{p_{j+1}}}}\\
 W_{q_{j}}\ar[r]_{S_{q_{j}l_{p_{j+1}}}}
 & W_{l_{p_{j+1}}}
}
\]

\noindent
commutes, by Equation~\eqref{for2}. Set
$T_{1}'=G_{p_{j+1}}S_{q_{j}l_{p_{j+1}}}$. Then the diagram

\[
\xymatrix{
& V_{g_{p_{j+1}}}\ar[rd]^{F_{g_{p_{j+1}}}}&
\\
W_{q_{j}}\ar[rr]_{S_{q_{j}l_{p_{j+1}}}}\ar[ru]^{T_{1}'}& &W_{l_{p_{j+1}}}
}
\]

\noindent commutes. Since $g$ is a premorphism we have
\begin{equation*}
T_{1}'S_{f_{n_{2}}q_{j}}=
G_{p_{j+1}}S_{q_{j}{p_{j+1}}}S_{f_{n_{2}}q_{j}}
=G_{p_{j+1}}S_{f_{n_{2}}p_{j+1}}
=E_{h_{n_{2}}g_{p_{j+1}}}G_{f_{n_{2}}}.
\end{equation*}
Hence, the following diagram is commutative:

\[
\xymatrix{
V_{r_{\resizebox{.28em}{.28em}{1}}}\ar[rr]^{E_{r_{\resizebox{.26em}{.26em}{1}}h_{n_{\resizebox{.26em}{.26em}{2}}}}}
\ar[rd]_{R_{1}'}& &
V_{h_{n_{\resizebox{.26em}{.26em}{2}}}}
\ar[rr]^{E_{h_{n_{\resizebox{.26em}{.26em}{2}}}g_{p_{j+1}}}} & &
V_{g_{p_{j+1}}}\ar[rd]^{F_{g_{p_{j+1}}}}\\
&W_{f_{n_{\resizebox{.26em}{.26em}{2}}}}
\ar[rr]_{S_{f_{n_{\resizebox{.26em}{.26em}{2}}}q_{j}}}
\ar[ru]_{G_{f_{n_{\resizebox{.26em}{.26em}{2}}}}}& &
W_{q_{j}}\ar[ur]_{T_{1}'}\ar[rr]_{S_{q_{j}l_{p_{j+1}}}}& &W_{l_{p_{j+1}}}
\ .}
\]

Set $t_{1}=q_{j}$ and $R_{1}=S_{f_{n_{2}}q_{j}}R_{1}'$. Then the diagram

\[
\xymatrix{
V_{r_{\resizebox{.28em}{.28em}{1}}}
\ar[rr]^{E_{r_{\resizebox{.26em}{.26em}{1}}h_{n_{\resizebox{.26em}{.26em}{2}}}}}
\ar[rd]_{R_{1}}& &
 V_{g_{p_{j+1}}}\ar[rd]^{F_{g_{p_{j+1}}}} &\\
&W_{t_{\resizebox{.28em}{.28em}{1}}}\ar[rr]_{S_{t_{1}l_{p_{j+1}}}}\ar[ru]_{T_{1}'}&
&W_{l_{p_{j+1}}}
}
\]

\noindent is commutative. Continuing this procedure, we obtain
sequences $(r_{k})_{k=1}^{\infty}$ and $(t_{k})_{k=1}^{\infty}$ of
positive integers with
$r_{1}<t_{1}<r_{2}<t_{2}<\cdots$, and   multiplicity matrices
$R_{k}: V_{r_{k}}\to W_{t_{k}}$ and $T_{k}: W_{t_{k}}\to V_{r_{k+1}}$, for each $k\geq 1$,
such that all the  diagrams in Lemma~\ref{lemeq} commute.
 In fact, by the construction, for each $k\geq 1$ there are
 positive integers $x_{k}$ and $y_{k}$ with $r_{k}\leq x_{k}\leq t_{k}$ and
 $t_{k}\leq y_{k}\leq r_{k+1}$ such that
 $R_{k}=S_{f_{x_{k}}t_{k}}F_{x_{k}}E_{r_{k}x_{k}}$ and
 $T_{k}=E_{g_{y_{k}r_{k+1}}}G_{y_{k}}S_{t_{k}y_{k}}$.
Therefore by
 Lemma~\ref{lemeq}, $\mathcal{A}_{1}\cong \mathcal{A}_{2}$.

To show that the classification functor $\mathcal{B}: \mathbf{AF} \to \mathbf{BD}$
 is strong, note that the $*$-isomorphism
 $\varphi : A\to B$ given by Lemma~\ref{lemeq} satisfies
$\varphi (\bigcup_{n=1}^{\infty}A_{n})\subseteq \bigcup_{n=1}^{\infty}B_{n}$, and
therefore $f$ above is admissible in Definition~\ref{deffu}, so that  $\mathcal{B}(\varphi)=[f]$.

An alternative proof can be given using \cite[Theorem~3]{el10} as follows.
By Theorem~\ref{throutbar} we have
$\mathcal{B}=\overline{\mathcal{B}}\,\overline{\mathcal{F}}$ and
$\overline{\mathcal{B}}$ is a strong classification functor, and
by \cite[Theorem~3]{el10} so also is $\overline{\mathcal{F}}$.
This shows that $\mathcal{B}$ is the composition of two strong classification functors and so
it is also a strong classification functor.
\end{proof}
\begin{corollary}\label{coreq}
Let $\mathcal{A}_{1}, \mathcal{A}_{2}\in \mathbf{AF}$. Then
$\mathcal{A}_{1}\cong \mathcal{A}_{2}$ in $\mathbf{AF}$ if and only if
$\mathcal{B}(\mathcal{A}_{1})\cong \mathcal{B}(\mathcal{A}_{2})$
in $\mathbf{BD}$.
\end{corollary}
\begin{proof}
Since
 $\mathcal{B}: \mathbf{AF} \to \mathbf{BD}$  is a functor,
$\mathcal{A}_{1}\cong \mathcal{A}_{2}$ implies
$\mathcal{B}(\mathcal{A}_{1})\cong \mathcal{B}(\mathcal{A}_{2})$.
The converse follows from Theorem~\ref{thrst}.
\end{proof}
As an application of  Theorem~\ref{thrst},
 we can give a proof of Glimm's theorem. Recall that, with  the notation of
Definition~\ref{defaf},  an AF algebra
 $\mathcal{A}=(A, (A_{n})_{n=1}^{\infty},(\varphi_{n})_{n=1}^{\infty})$
 in $\mathbf{AF}$ is called
 a \emph{UHF~algebra} if $A$ is unital, each $A_{n}$ contains the unit of $A$, and
 each $A_{n}$ is a simple C$^{*}$-algebra; thus $A_{n}\cong\mathcal{M}_{k_{n}}$,
 for some $k\geq 1$. Let
$ \mathcal{B}(\mathcal{A})=(V,E)$; thus $V_{n}=(k_{n})$.
 According to Theorem~\ref{thrfd}, since the $*$-isomorphism
 $\varphi_{n+1}\varphi_{n}^{-1}: \mathcal{M}_{k_{n}}\to \mathcal{M}_{k_{n+1}}$
 is unital, we have $k_n | k_{n+1}$ and
 $E_{n}=R_{\varphi_{n+1}\varphi_{n}^{-1}}=k_{n+1}/k_{n}$.
 Therefore, the Bratteli diagram of $\mathcal{A}$ is independent of the choice
 of $\varphi_{n}$'s.

 Denote by $\Prime$ the set of all prime numbers. Define
 $\varepsilon_{\mathcal{A}}:\Prime \to \N \cup\{0,\infty\}$ with
 \begin{equation*}
 \varepsilon_{\mathcal{A}}(p)=\sup\{ m\geq 0 :
 p^{m}| k_{n},\ \text{for some}\ n\geq 1\}, \ p\in \Prime .
 \end{equation*}
The following famous result of Glimm is an easy
consequence  of Proposition~\ref{profun} and  Theorem~\ref{thrst}.
 \begin{theorem}[Glimm,  \cite{gl60}, Theorem~1.12]\label{thrgl}
 Let $\mathcal{A}_{1}, \mathcal{A}_{2}\in \mathbf{AF}$ be two UHF~algebras.
 Then  $\mathcal{A}_{1}\cong \mathcal{A}_{2}$ if and only if
 $\varepsilon_{\mathcal{A}_{1}}=\varepsilon_{\mathcal{A}_{2}}$.
 \end{theorem}
\begin{proof}
Let
$\mathcal{A}_{1}=(A, (A_{n})_{n=1}^{\infty},(\varphi_{n})_{n=1}^{\infty})$ and
$\mathcal{A}_{2}=(B, (B_{n})_{n=1}^{\infty},(\psi_{n})_{n=1}^{\infty})$.
Write $\mathcal{B}(\mathcal{A}_{1})=(V,E)$ and
$ \mathcal{B}(\mathcal{A}_{2})=(W,S)$. According to the remarks preceding this theorem,
there are sequences $(k_{n})_{n=1}^{\infty}$ and
$(m_{n})_{n=1}^{\infty}$ of natural numbers such that
$V_{n}=(k_{n})$ and $W_{n}=(m_{n})$, for each $n\geq 1$, and
$k_{1}|k_{2}|\cdots$ and $m_{1}|m_{2}|\cdots$.
Note that $\varepsilon_{\mathcal{A}_{1}}=\varepsilon_{\mathcal{A}_{2}}$
if and only if
\begin{equation}\label{condi}
\forall n\geq 1\  \exists l\geq 1\ \ \  k_{n} |m_{l} \ \ \ \ \text{and}\ \ \ \
\forall l\geq 1\  \exists n\geq 1\ \ \  m_{l} | k_{n}.
\end{equation}

Suppose that $\mathcal{A}_{1}\cong \mathcal{A}_{2}$. Then
by Proposition~\ref{profun},
$\mathcal{B}(\mathcal{A}_{1})\cong \mathcal{B}(\mathcal{A}_{2})$, i.e.,
 there are premorphisms
$f: \mathcal{B}(\mathcal{A}_{1})\to \mathcal{B}(\mathcal{A}_{2})$
and $g: \mathcal{B}(\mathcal{A}_{2})\to \mathcal{B}(\mathcal{A}_{1})$
such that $gf\sim \mathrm{id}_{\mathcal{B}(\mathcal{A}_{1})}$ and
$fg\sim \mathrm{id}_{\mathcal{B}(\mathcal{A}_{2})}$, and Condition~(\ref{condi}) follows.

Now suppose that $\varepsilon_{\mathcal{A}_{1}}=\varepsilon_{\mathcal{A}_{2}}$;
thus, Condition~(\ref{condi}) is satisfied. In other words, there
are strictly increasing sequences $(f_{n})_{n=1}^{\infty}$ and $(g_{n})_{n=1}^{\infty}$ of non-zero positive integers
such that $k_{n} | m_{f_{n}}$ and $m_{n} | k_{g_{n}}$, for each
$n\geq 1$. Define premorphisms
$f: \mathcal{B}(\mathcal{A}_{1})\to \mathcal{B}(\mathcal{A}_{2})$
and $g: \mathcal{B}(\mathcal{A}_{2})\to \mathcal{B}(\mathcal{A}_{1})$
with $f=((F_{n})_{n=1}^{\infty},(f_{n})_{n=1}^{\infty})$ and
$g=((G_{n})_{n=1}^{\infty},(g_{n})_{n=1}^{\infty})$,
where $F_{n}=(m_{f_{n}}/k_{n})$ and
$G_{n}=(k_{g_{n}}/m_{n})$. It is easy to see that
$gf\sim \mathrm{id}_{\mathcal{B}(\mathcal{A}_{1})}$ and
$fg\sim \mathrm{id}_{\mathcal{B}(\mathcal{A}_{2})}$; thus,
$\mathcal{B}(\mathcal{A}_{1})\cong \mathcal{B}(\mathcal{A}_{2})$.
Therefore by Theorem~\ref{thrst} we have
$\mathcal{A}_{1}\cong \mathcal{A}_{2}$.

Alternatively (not using Theorem~\ref{thrst}---but the
functorial property of Proposition~\ref{profun} is still used in the first half
of the theorem), $A$ and $B$ can be seen each to have the structure of an infinite
tensor product of matrix algebras of prime order, with the multiplicities
of the primes determined by the Bratteli diagram data, from which isomorphism
is immediate if the data is the same.
\end{proof}
\section{A Homomorphism Theorem}
Let us consider further  the  properties of the functor
$\mathcal{B}: \mathbf{AF} \to \mathbf{BD}$.
The following result  may be considered as a generalization of part of Theorem~\ref{thrst}.
Theorem~\ref{thrst} says that isomorphisms in the codomain category can be lifted back to isomorphisms in the domain category, and in particular to homomorphisms. The following theorem
states this for arbitrary homomorphisms. (Theorem~\ref{thrst} cannot be deduced immediately from
Theorem~\ref{thrful}; but see \cite{el10}.)
(The proofs of the two theorems are
similar: roughly speaking, a two-sided and a one-sided intertwining
argument.)

\begin{theorem}\label{thrful}
The functor $\mathcal{B}: \mathbf{AF} \to \mathbf{BD}$ is
a full functor in the sense that if
 $\mathcal{A}_{1},\mathcal{A}_{2}\in\mathbf{AF}$ and
$f: \mathcal{B}(\mathcal{A}_{1})\to \mathcal{B}(\mathcal{A}_{2})$
is a morphism in $\mathbf{BD}$,
then there is a morphism $\varphi : \mathcal{A}_{1}\to \mathcal{A}_{2}$
 in $\mathbf{AF}$
such that $\mathcal{B}(\varphi)=f$.
\end{theorem}
\begin{proof}
Let $\mathcal{A}_{1},\mathcal{A}_{2}\in\mathbf{AF}$ and
$f: \mathcal{B}(\mathcal{A}_{1})\to \mathcal{B}(\mathcal{A}_{2})$
be a morphism in $\mathbf{BD}$. Write
$\mathcal{A}_{1}=(A, (A_{n})_{n=1}^{\infty},(\varphi_{n})_{n=1}^{\infty})$,
$\mathcal{A}_{2}=(B, (B_{n})_{n=1}^{\infty},(\psi_{n})_{n=1}^{\infty})$,
$\mathcal{B}(\mathcal{A}_{1})=((V_{n})_{n=1}^{\infty}, (E_{n})_{n=1}^{\infty})$, and
$\mathcal{B}(\mathcal{A}_{2})=((W_{n})_{n=1}^{\infty}, (S_{n})_{n=1}^{\infty}))$.
Suppose that $f$ is the equivalence class of the premorphism
$((F_{n})_{n=1}^{\infty},(f_{n})_{n=1}^{\infty})$, according to
Definition~\ref{defpre}.
Thus, each $F_{n}$ is a  multiplicity  matrix from $V_{n}$ to $W_{f_{n}}$ and the following diagram commutes:

\[
\xymatrix{V_{1}\ar[r]^{E_{1}}\ar[d]_{F_{1}}
 &V_{2}\ar[r]^-{E_{2}}\ar[d]_{F_{2}} &V_{3}\ar[r]^-{E_{3}}\ar[ld]^{F_{3}} &\cdots \\
 W_{1}\ar[r]_{S_{1}}
 &W_{2}\ar[r]_{S_{2}} &W_{3}\ar[r]_{S_{3}}&\cdots
\ . }
\]

By Theorem~\ref{thrfd},
there is a $*$-homomorphism
$g_{n}:A_{n}\to B_{f_{n}}$ with multiplicity matrix  $F_{n}$, i.e.,
$R_{g_{n}}=F_{n}$, $n\geq 1$, and we have the following (a priori non-commutative) diagram:

\[
\xymatrix{A_{1}\ar@{^{(}->}[r]\ar[d]_{g_{1}}
 &A_{2}\ar@{^{(}->}[r]\ar[d]_{g_{2}} &A_{3}\ar@{^{(}->}[r]\ar[ld]^{g_{3}} &\cdots \\
 B_{1}\ar@{^{(}->}[r]
 &B_{2}\ar@{^{(}->}[r] &B_{3}\ar@{^{(}->}[r]&\cdots\ .
 }
\]

\noindent Using Corollary~\ref{corcom},
we can replace $g_{2}$ with $(\mathrm{Ad}\,{u_{2}})g_{2}$,
for some unitary $u_{2}\in B_{2}$,
such that the first left square is commutative. Since unitaries do not change  multiplicity matrices (Lemma~\ref{lemun}),
one can continue this procedure to obtain unitaries
$(u_{n})_{n\geq 2}$ such that the above diagram is commutative when each $g_{n}$
is replaced by $(\mathrm{Ad}\,{u_{n}})g_{n}$ ($n\geq 2$). Therefore there is a
$*$-homomorphism $\varphi : A\to B$ such that
$\varphi\upharpoonright_{A_{1}}=g_{1}$
and $\varphi\upharpoonright_{A_{n}}=(\mathrm{Ad}\,{u_{n}})g_{n}$,
for each $n\geq 2$. Using Lemma~\ref{lemun},
we have $R_{(\mathrm{Ad}\,{u_{n}})g_{n}}=F_{n}$, $n\geq 2$. Therefore,
 $\mathcal{B}(\varphi)=f$.
\end{proof}
\begin{proposition}\label{proonto}
Let  $B=((V_{n})_{n=1}^{\infty},(E_{n})_{n=1}^{\infty})$
be a  Bratteli diagram. Then there is an
$\mathcal{A}\in\mathbf{AF}$ such that
$ \mathcal{B}(\mathcal{A})=B$.
\end{proposition}
\begin{proof}
By Definition~\ref{defbd},  each $E_{n}$ is an embedding matrix.
Set $A_{n}'=C^{*}(V_{n})$ and $\varphi_{n}'=h(E_{n})$, for each $n\geq 1$.
 (See the remark following
Lemma~\ref{lempr}
for the notations  $C^{*}(V_{n})$ and $h(E_{n})$.)
By Theorem~\ref{thrfd},
$h_{n}: A_{n}'\to A_{n+1}'$ is injective. Now set
$A=\varinjlim (A_{n}',h_{n})$ and let $\alpha^{n}:A_{n}'\to A$
 denote the $*$-homomorphism that comes
from the construction of the direct limit, $n\geq 1$.
Set $A_{n}=\alpha^{n}(A_{n}')$ and denote by $\varphi_{n}: A_{n}\to A_{n}'$
 the inverse of $\alpha^{n}:A_{n}'\to A_{n}$ (which exists, since each
$h_{n}$ is injective). Now
$\mathcal{A}=(A, (A_{n})_{n=1}^{\infty},(\varphi_{n})_{n=1}^{\infty})
\in\mathbf{AF}$ and
$ \mathcal{B}(\mathcal{A})=B$.
\end{proof}
 Let us denote by $\mathbf{AF}_{1}$ the subcategory of $\mathbf{AF}$
whose objects are unital AF~algebras and whose morphisms are unital
homomorphisms; more precisely,
$(A, (A_{n})_{n=1}^{\infty},(\varphi_{n})_{n=1}^{\infty})\in\mathbf{AF} $
is an object of  $\mathbf{AF}_{1}$ if $A$ is unital and each $A_{n}$
contains the unit of $A$. The next proposition follows from part (2) of Theorem~\ref{thrfd}.

\begin{proposition}
Let $\mathcal{A}\in\mathbf{AF}$ and
$ \mathcal{B}(\mathcal{A})=(V,E)$. Then
$\mathcal{A}$ is in $\mathbf{AF}_{1}$ if and only if $E_{n}V_{n}=V_{n+1}$,
for each $n\in \N$.
\end{proposition}

Next we give a criterion to check if the functor
$\mathcal{B}$ sends two AF~algebras to the same diagram. The proof is straightforward.

\begin{proposition}\label{proinj}
Let $\mathcal{A}_{1}$ and $\mathcal{A}_{2}$ be in $\mathbf{AF}$ with
$\mathcal{A}_{1}=(A, (A_{n})_{n=1}^{\infty},(\varphi_{n})_{n=1}^{\infty})$ and
$\mathcal{A}_{2}=(B, (B_{n})_{n=1}^{\infty},(\psi_{n})_{n=1}^{\infty})$.
Then
$\mathcal{B}(\mathcal{A}_{1})=\mathcal{B}(\mathcal{A}_{2})$ if and only if
there is a $*$-isomorphism
$\varphi :A\to B$ such that $\varphi(A_{n})=B_{n}$,  and
the  multiplicity matrix of
$\varphi\upharpoonright_{A_{n}}: A_{n} \to B_{n}$
is the identity, for each $n\geq 1$.
\end{proposition}

The following theorem (essentially due to Bratteli) gives a combinatorial criterion for isomorphism
of Bratteli diagrams.
\begin{theorem}\label{thrcong}
Let $B=(V,E)$ and $C=(W,S)$ be two Bratteli diagrams. Then $B\cong C$ in $\mathbf{BD}$
 if and only if there is a third Bratteli diagram $D=(Z,T)$, which is constructed from two subsequences of $B$ and $C$ as follows.
There are positive integers $(r_{k})_{k=1}^{\infty}$ and $(t_{k})_{k=1}^{\infty}$
 with
$r_{1}<t_{1}<r_{2}<t_{2}<\cdots$
such that  $Z_{2k-1}=V_{r_{k}}$, $Z_{2k}=W_{t_{k}}$,
$T_{2k-1,2k+1}=E_{r_{k}r_{k+1}}$, and
$T_{2k,2k+2}=S_{t_{k}t_{k+1}}$, for each $k\geq 1$, i.e.,
the following diagram is commutative:
\end{theorem}

\[
\xymatrix{
 V_{r_{\resizebox{.28em}{.28em}{1}}}
 \ar[r]^{T_{1}}
 \ar@/^2pc/[rr]^{E_{r_{\resizebox{.26em}{.26em}{1}}r_{\resizebox{.26em}{.26em}{2}}}}&
 W_{t_{\resizebox{.28em}{.28em}{1}}}
 \ar@/_2pc/[rr]_{S_{t_{\resizebox{.26em}{.26em}{1}}t_{\resizebox{.26em}{.26em}{2}}}}
 \ar[r]^{T_{2}}
 & V_{r_{\resizebox{.28em}{.28em}{2}}}
 \ar@/^2pc/[rr]^{E_{r_{\resizebox{.26em}{.26em}{2}}r_{\resizebox{.26em}{.26em}{3}}}}
 \ar[r]^{T_{3}} &
 W_{t_{\resizebox{.28em}{.28em}{2}}}
 \ar@/_2pc/[rr]_{S_{t_{\resizebox{.26em}{.26em}{2}}t_{\resizebox{.26em}{.26em}{3}}}}
 \ar[r]^{T_{4}} &
 V_{r_{\resizebox{.28em}{.28em}{3}}}
 \ar[r]^{T_{5}} &
 W_{t_{\resizebox{.28em}{.28em}{3}}}\ar[r] & \cdots\ .
 }
\]
\begin{proof}
Choose
$\mathcal{A}_{1}$ and $\mathcal{A}_{2}$ in $\mathbf{AF}$ such that
$\mathcal{B}(\mathcal{A}_{1})=B$ and
 $\mathcal{B}(\mathcal{A}_{2})=C$, as in Proposition~\ref{proonto}.
 We have $B\cong C$ if and only if
 $\mathcal{A}_{1}\cong \mathcal{A}_{2}$, by Corollary~\ref{coreq}.
 Now the statement follows from
Lemma~\ref{lemeq}, on observing that in the proof of this lemma, the  multiplicity matrices $R_{k}$ and $T_{k}$  are
indeed embedding matrices.
\end{proof}
\begin{proposition}\label{prolift}
Suppose that $\varphi: \mathcal{A}_{1}\to \mathcal{A}_{2}$ is a morphism in
$\mathbf{AF}$ and
$((F_{n})_{n=1}^{\infty},(f_{n})_{n=1}^{\infty})$
 is an arbitrary premorphism whose equivalence class is
 $\mathcal{B}(\varphi)$. Write
 $\mathcal{B}(\mathcal{A}_{1})=(V,E)$ and
 $\mathcal{B}(\mathcal{A}_{2})=(W,S)$.
 \begin{itemize}
 \item[(1)]
 $\varphi$ is injective if and only if each $F_{n}$ is an embedding matrix;
 \item[(2)]
 if $\mathcal{A}_{1},\mathcal{A}_{2}\in \mathbf{AF}_{1}$, then
 $\varphi$ is unital if and only if $F_{n}V_{n}=W_{f_{n}}$, $n\geq 1$.
 \end{itemize}
\end{proposition}
\begin{proof}
First suppose that
$f=((F_{n})_{n=1}^{\infty},(f_{n})_{n=1}^{\infty})$
is the premorphism associated to $\varphi$, as in Definition~\ref{deffu}.
In this case, the statements (1) and (2) follow from the parts (1) and (2) in Theorem~\ref{thrfd}.
Now suppose that
$f=((F_{n})_{n=1}^{\infty},(f_{n})_{n=1}^{\infty})$
is an arbitrary premorphism,  the equivalence class of which is
 $\mathcal{B}(\varphi)$. Let
$g=((G_{n})_{n=1}^{\infty},(g_{n})_{n=1}^{\infty})$
  be a premorphism associated to $\varphi$, as in Definition~\ref{deffu}.
 Applying
 Definition~\ref{defeq2} and Proposition~\ref{proeqpre},
 we see that the statements (1) and (2) hold for $F_{n}$ if and only if
 they hold for $G_{n}$, which they do as shown.
\end{proof}
\section{Relations with Abstract Classifying Categories}
In this section, let us investigate the relation between the following three classifying categories for
AF~algebras: the category of Bratteli diagrams $\mathbf{BD}$,
the abstract category
$\mathbf{AF}^{\mathrm{out}}$, and
the abstract category
$\overline{\mathbf{AF}^{\mathrm{out}}}$ introduced in \cite{el10}.
In the next section we will also consider the category
of dimension groups $\mathbf{DG}$ introduced in \cite{el93}.

In particular, now and in the next section we shall show that the three categories
 $\mathbf{BD}$, $\overline{\mathbf{AF}^{\mathrm{out}}}$, and
 $\mathbf{DG}$
are all equivalent, and hence are classifying categories for each other.

Let us also investigate the relation between the strong classification functors
$\mathcal{B}:\mathbf{AF}\to \mathbf{BD}$,
$\mathcal{F} :\mathbf{AF}\to \mathbf{AF}^{\mathrm{out}}$, and
$\overline{\mathcal{F}} :\mathbf{AF}\to  \overline{\mathbf{AF}^{\mathrm{out}}}$---and in the next section,
 the relation to the functor
$\mathrm{K}_{0}: \mathbf{AF}\to \mathbf{DG}$.

The following lemma may be considered as part of the literature (basically due to Glimm---see below); we  give a proof anyway (cf.~Lemma~4.2 and Theorem~4.3 of \cite{el93}).
\begin{lemma}\label{lemunaf}
For each $\varepsilon>0$ there is a $\delta>0$ such that if
 $A$ is a unital C*-algebra, $B$ is a C*-subalgebra of $A$ containing the unit
of $A$, and $u$ is a unitary of $A$ with $d(u,B)<\delta$, then
there is a unitary $v\in B$ such that $\Vert u-v\Vert<\varepsilon$.
\end{lemma}
\begin{proof}
The statement follows from \cite[Lemma~1.9]{gl60}. In fact, in the proof
of \cite[Lemma~1.9]{gl60}, Glimm does not use the assumption of orthogonality of projections. Thus putting $E_1 = E_2 = F_1 = F_2 = 1$ in that lemma, the
statement follows.
There is also a direct proof as follows.
Set $\delta=\min\{\varepsilon, \frac{1}{8}\}$.
Let $A$ be a unital C*-algebra, $B$ be a C*-subalgebra of $A$
containing the unit
of $A$, and $u$ be a unitary of $A$ with $d(u,B)<\delta$.
Thus there is $a\in B$ such that $\Vert u-a\Vert <\delta$. Thus $a$ is invertible.
Set $v=(aa^{*})^{-\frac{1}{2}}a$; hence  $vv^{*}=v^{*}v=1$ and  $v$  is a unitary
in $B$. We have
\begin{align*}
\Vert aa^{*}-1 \Vert
&\leq \Vert aa^{*}- au^{*}\Vert+
\Vert au^{*}-uu^{*} \Vert\\
&\leq
\Vert a\Vert \Vert a^{*}- u^{*}\Vert+
\Vert a-u \Vert\\
&\leq
(\Vert a\Vert+1)\Vert a-u \Vert\\
&\leq
3\Vert a-u \Vert<\frac{1}{2}.
\end{align*}
Thus $\Vert (aa^{*})^{-1}\Vert<2$. Using functional calculus we have
\begin{align*}
\Vert (aa^{*})^{-\frac{1}{2}}-1 \Vert
&\leq \Vert (aa^{*})^{-1}-1 \Vert\\
&\leq \Vert (aa^{*})^{-1}\Vert \Vert aa^{*}-1 \Vert\\
&\leq 2 \Vert aa^{*}-1 \Vert\\
&\leq 6 \Vert a-u \Vert<1.
\end{align*}
Therefore we have
\begin{align*}
\Vert u-v\Vert &\leq \Vert u- (aa^{*})^{-\frac{1}{2}}u\Vert+
\Vert (aa^{*})^{-\frac{1}{2}}u-v \Vert\\
&\leq \Vert 1- (aa^{*})^{-\frac{1}{2}} \Vert +
\Vert (aa^{*})^{-\frac{1}{2}} \Vert \Vert u-a \Vert\\
&\leq 6 \Vert a-u \Vert +2 \Vert a-u \Vert\\
&\leq 8 \Vert a-u \Vert<\varepsilon. \qedhere
\end{align*}
\end{proof}
The following immediate consequence of Lemma~\ref{lemunaf}  enables us to approximate the unitaries of an AF~algebra by
unitaries of an increasing sequence of C*-subalgebras with dense union. We will use this statement in the
proof of Lemma~\ref{lemad}.
\begin{corollary}\label{corunaf}
Let $A=\overline{\bigcup_{n\geq  1}A_{n}}$ be a unital AF algebra where
$(A_{n})_{n=1}^{\infty}$ is an increasing sequence of finite dimensional  C$^{*}$-subalgebras
of $A$ each containing the unit of $A$.
Then $u\in A$ is a unitary if and only if there are unitaries $(u_{n})_{n=1}^{\infty}$
such that $u_{n}\in A_{n}$, $n\geq 1$, and $u_{n}\longrightarrow u$.
\end{corollary}
We shall need the following  technical lemma in the proof of  Lemma~\ref{lemeqbh} and
Corollary~\ref{cordis}.
\begin{lemma}\label{lemad}
Let $B=\overline{\bigcup_{n\geq  1}B_{n}}$ be an AF~algebra where
$(B_{n})_{n=1}^{\infty}$ is an increasing sequence of finite dimensional  C$^{*}$-subalgebras
of $B$. Let $A$ be a finite dimensional  C$^{*}$-algebra and let
$ \varphi,\psi :A\to B$ be  $*$-homomorphisms such that $\Vert  \varphi-\psi\Vert<1$.
Let
$u,v$ be unitaries in $B^{\sim}$ such that
there are positive integers $n$ and $m$ such that
$u \varphi(A)u^{*}\subseteq B_{n}$ and
$v\psi(A)v^{*}\subseteq B_{m}$.
Then there is a positive integer $k\geq n,m$
such that the  $*$-homomorphisms
$(\mathrm{Ad}\,{u}) \varphi, (\mathrm{Ad}\,{v})\psi :A\to B_{k}$ have the  same
 multiplicity matrices, i.e.,
$R_{(\mathrm{Ad}\,{u}) \varphi}=R_{(\mathrm{Ad}\,{v})\psi}$.
\end{lemma}
\begin{proof}
Let 1 denote the unit of $B^{\sim}$.  Set $B_{n}^{\sim}=B_{n}+\C 1$,
so that $B^{\sim}=\overline{\bigcup_{n\geq  1}B_{n}^{\sim}}$.
By Corollary~\ref{corunaf}, there  is a positive integer $k\geq n,m$
and a unitary $w\in B_{k}^{\sim}$ such that
$\Vert uv^{*}-w\Vert<\frac{1}{2}(1 -\Vert \varphi-\psi\Vert)$.
We have the following (a priori non-commutative) diagram:

\[
\xymatrix{A\ar[d]_{(\mathrm{Ad}\,{u})\varphi}\ar[r]^{(\mathrm{Ad}\,{v})\psi}
  & B_{m}\ar[d]^{\mathrm{Ad}\,{w}} \\
 B_{n}\ar@{^{(}->}[r]
 &B_{k}\ .
 }
\]

\noindent Consider the two $*$-homomorphisms
$(\mathrm{Ad}\,{w})(\mathrm{Ad}\,{v})\psi,(\mathrm{Ad}\,{u})\varphi: A\to B_{k}$. We have

\[
\Vert (\mathrm{Ad}\,{w})(\mathrm{Ad}\,{v})\psi-(\mathrm{Ad}\,{u})\varphi\Vert
=\Vert (\mathrm{Ad}\,{wv})\psi-(\mathrm{Ad}\,{u})\varphi\Vert
\leq
2\Vert wv-u\Vert +\Vert \varphi-\psi\Vert<1.
\]

\noindent Hence by Lemma~\ref{lempart},
$R_{(\mathrm{Ad}\,{w})(\mathrm{Ad}\,{v})\psi}=R_{(\mathrm{Ad}\,{u})\varphi}$.
Now define $\eta :B_{k}\to B_{k}$ with $\eta(x)=(\mathrm{Ad}\,{w})(x)$, $x\in B_{k}$.
Then there is a unitary $w'\in B_{k}$  such that $\eta=\mathrm{Ad}\,{w'}$.
In fact, $w\in B_{k}^{\sim}$ and thus $w=a+\lambda 1$, for some $a\in B_{k}$ and
$\lambda\in \C$. Set $w'=a+\lambda 1_{B_{k}}$. Since $w$ is a unitary, so is $w'$,
and we have $(\mathrm{Ad}\,{w})(x)=(\mathrm{Ad}\,{w'})(x)$, $x\in B_{k}$. Thus
 $\eta=\mathrm{Ad}\,{w'}$ and hence $R_{\eta}$ is the identity matrix, by Theorem~\ref{thrfd}.
 On the other hand,
 $(\mathrm{Ad}\,{w})(\mathrm{Ad}\,{v})\psi=\eta\circ (\mathrm{Ad}\,{v})\psi$, and so
 $R_{(\mathrm{Ad}\,{w})(\mathrm{Ad}\,{v})\psi}=R_{\eta\circ (\mathrm{Ad}\,{v})\psi}=
 R_{\eta}R_{(\mathrm{Ad}\,{v})\psi}=R_{(\mathrm{Ad}\,{v})\psi}$.
 Therefore we have $R_{(\mathrm{Ad}\,{u})\varphi}=R_{(\mathrm{Ad}\,{v})\psi}$.
\end{proof}
The functor $\mathcal{B}: \mathbf{AF} \to \mathbf{BD}$ is of course
not faithful (we follow  \cite{gb84, ma98} for categorical definitions).
The following gives  useful criteria to check  whether the images
under $\mathcal{B}$
of two morphisms of $\mathbf{AF}$
are equal in $\mathbf{BD}$. This enables us to make connections between
morphisms of $\mathbf{BD}$ and morphisms of the categories
$\mathbf{AF}^{\mathrm{out}}$ and
$\overline{\mathbf{AF}^{\mathrm{out}}}$ (see Theorems~\ref{throut} and
\ref{throutbar} below).
\begin{lemma}\label{lemeqbh}
Let $\mathcal{A}_{1},\mathcal{A}_{2}\in\mathbf{AF}$ where
$\mathcal{A}_{1}=(A, (A_{n})_{n=1}^{\infty},(\varphi_{n})_{n=1}^{\infty})$
and
$\mathcal{A}_{2}=(B, (B_{n})_{n=1}^{\infty},(\psi_{n})_{n=1}^{\infty})$.
Let
$\varphi ,\psi : \mathcal{A}_{1}\to \mathcal{A}_{2}$
be  morphisms in $\mathbf{AF}$. The following statements are equivalent:
\begin{itemize}
\item[(1)]
$\mathcal{B}(\varphi)=\mathcal{B}(\psi)$;
\item[(2)]
there is a sequence of unitaries
$(u_{n})_{n=1}^{\infty}$ in $B^{\sim}$ such that
$\varphi=(\mathrm{Ad}\,{u_{n}})\psi$ on $A_{n}$, $n\geq 1$;
\item[(3)]
there is a sequence of unitaries
$(u_{n})_{n=1}^{\infty}$ in $B^{\sim}$ such that
$\varphi(a)=\lim\limits_{n\to \infty}(\mathrm{Ad}\,{u_{n}})\psi(a)$, $a\in A$.
\end{itemize}
\end{lemma}
\begin{proof}
Choose a sequence of unitaries $(v_{n})_{n=1}^{\infty}$
in $B^{\sim}$,
and a sequence of positive integers $(f_{n})_{n=1}^{\infty}$,
as in Definition~\ref{deffu}, giving rise to a premorphism
$((F_{n})_{n=1}^{\infty}, (f_{n})_{n=1}^{\infty})$ with equivalence class
$\mathcal{B}(\varphi)$. Similarly, choose a sequence of unitaries $(w_{n})_{n=1}^{\infty}$
in $B^{\sim}$,
and a sequence of positive integers $(g_{n})_{n=1}^{\infty}$, giving rise
to a premorphism
$((G_{n})_{n=1}^{\infty}, (g_{n})_{n=1}^{\infty})$ for
$\mathcal{B}(\psi)$.

$(1)\Rightarrow (2)$: Suppose that $\mathcal{B}(\varphi)=\mathcal{B}(\psi)$. Hence
$((F_{n})_{n=1}^{\infty}, (f_{n})_{n=1}^{\infty})$ is equivalent to
$((G_{n})_{n=1}^{\infty}, (g_{n})_{n=1}^{\infty})$. Fix $n\geq 1$.
By Proposition~\ref{proeqpre}, there is an $m\geq f_{n},g_{n}$ such that
$S_{f_{n}m}F_{n}=S_{g_{n}m}G_{n}$, where
$S_{f_{n}m}$ and $S_{g_{n}m}$ are the  multiplicity matrices of the injections
$j_{1}: B_{f_{n}}\hookrightarrow B_{m}$  and
$j_{2}: B_{g_{n}}\hookrightarrow B_{m}$, respectively.
On the other hand,
$F_{n}$ and $G_{n}$ are the  multiplicity matrices of
$(\mathrm{Ad}\,{v_{n}})\varphi: A_{n}\to B_{f_{n}}$ and
$(\mathrm{Ad}\,{w_{n}})\psi: A_{n}\to B_{g_{n}}$, respectively, by Definition~\ref{deffu}.
Thus $R_{j_{1}(\mathrm{Ad}\,{v_{n}})\varphi}=S_{f_{n}m}F_{n}=S_{g_{n}m}G_{n}=
R_{j_{2}(\mathrm{Ad}\,{w_{n}})\psi}$. By Lemma~\ref{lemun},
there is a unitary $u\in B_{m}$ such that
$j_{1}(\mathrm{Ad}\,{v_{n}})\varphi=(\mathrm{Ad}\,{u})j_{2}(\mathrm{Ad}\,{w_{n}})\psi$
on $A_{n}$.
Set $w=u-1_{B_{m}}+1$, where 1 is the unit of $B^{\sim}$.
One can easily see that $w$ is a unitary in $B^{\sim}$ and again we have
$j_{1}(\mathrm{Ad}\,{v_{n}})\varphi=(\mathrm{Ad}\,{w})j_{2}(\mathrm{Ad}\,{w_{n}})\psi$
on $A_{n}$.
Set $u_{n}=v_{n}^{*}ww_{n}$. Therefore
$\varphi=(\mathrm{Ad}\,{u_{n}})\psi$ on $A_{n}$.

$(2)\Rightarrow (3)$: This holds as $(A_{n})_{n=1}^{\infty}$ is increasing
with union dense in $A$.

$(3)\Rightarrow (1)$: Suppose that there is a sequence of unitaries
$(u_{n})_{n=1}^{\infty}$ in $B^{\sim}$ such that
$\varphi$ is the pointwise limit of the sequence
$((\mathrm{Ad}\,{u_{n}})\psi)_{n=1}^{\infty}$ on $A$.
Fix $n\geq 1$. Since $((\mathrm{Ad}\,{u_{m}})\psi)_{m=1}^{\infty}$ converges
$\varphi$ on compact subsets of $A$ and the unit ball of $A_{n}$ is compact,
$\Vert (\mathrm{Ad}\,{u_{m}})\psi- \varphi\Vert_{A_{n}}\longrightarrow 0$, as $m$ tends to infinity.
Thus there is an $n'\geq 1$ such that
$\Vert (\mathrm{Ad}\,{u_{n'}})\psi- \varphi\Vert_{A_{n}}<1$.
Set $u=v_{n}$ and $v=w_{n}u_{n'}^{*}$. Hence
$u \varphi(A_{n})u^{*}\subseteq B_{f_{n}}$ and
$v(\mathrm{Ad}\,{u_{n'}})\psi(A_{n})v^{*}= w_{n}\psi(A_{n})w_{n}^{*}\subseteq B_{g_{n}}$.
Applying Lemma~\ref{lemad}, there is an $m\geq f_{n},g_{n}$
such that
$(\mathrm{Ad}\,{u})\varphi, (\mathrm{Ad}\,{v})\circ ((\mathrm{Ad}\,{u_{n'}})\psi) :A_{n}\to B_{m}$
have the  same
 multiplicity matrices; that is,
$(\mathrm{Ad}\,{v_{n}})\varphi, (\mathrm{Ad}\,{w_{n}})\psi:A_{n}\to B_{m}$
have the  same
 multiplicity matrices. By Proposition~\ref{proeqpre}, the premorphisms
$((F_{n})_{n=1}^{\infty}, (f_{n})_{n=1}^{\infty})$ and
$((G_{n})_{n=1}^{\infty}, (g_{n})_{n=1}^{\infty})$ are  equivalent  and therefore
$\mathcal{B}(\varphi)=\mathcal{B}(\psi)$.
\end{proof}
\begin{remark}
Lemma~\ref{lemeqbh} remains valid if we replace $B^{\sim}$ with $B^{+}$.
Since, if $B$ is non-unital, we have  $B^{\sim}=B^{+}$, and if $B$ is unital, according to our convention,
$B^{\sim}=B$;
thus by the techniques applied in the proof of Lemma~\ref{lemad} and
Lemma~\ref{lemeqbh} for interchanging the unitaries of $B$ and $B^{+}$,
the statement is also true for $B^{+}$ instead of $B^{\sim}$.
\end{remark}
In general, the sequence of unitaries in Lemma~\ref{lemeqbh} cannot be replaced by a single unitary.
In other words, in the setting of that lemma, the condition
$\mathcal{B}(\varphi)=\mathcal{B}(\psi)$ does not necessarily imply that there is a unitary
$u\in B^{\sim}$ with $\varphi =(\mathrm{Ad}\,{u})\psi$. (See the following
example.)
\begin{example}\label{exaun}
Consider the C$^*$-algebra $A=\mathcal{K}(l^{2})$
and let $(e_{n})_{n=1}^{\infty}$ be an orthonormal basis for $l^2$.
Consider  the C$^{*}$-subalgebra  $A_{n}$ generated by the rank one operators
$\{ e_{i}\otimes e_{j}^{*} | 1\leq i, j\leq n\}$ for $n\geq 1$. Then
$A_{1}\subseteq A_{2}\subseteq \cdots$, and $A=\overline{\bigcup_{n\geq 1}A_{n}}$.
Define $\varphi, \psi : A\to A$ as follows. Set $\psi=\mathrm{id}_{A}$. For each $n\in \N$,
let $u_{n}$ denote the unitary in $A^{\sim}=\mathcal{K}(l^{2})\oplus I$ defined by
$u_{n}(e_{k})=e_{k+1}$, for $1\leq k\leq n$, $u_{n}(e_{n+1})=e_{1}$, and
$u_{n}(e_{k})=e_{k}$, for $k\geq n+2$. Then
$\mathrm{Ad}\,{u_{n}}$ and $\mathrm{Ad}\,{u_{m}}$ agree on $A_{n}$ when
$n\leq m$. Set $\varphi=\mathrm{Ad}\,{u_{n}}$ on $A_{n}$, $n\geq 1$.
Then $\varphi: A\to A$ is a $*$-homomorphism and
$\varphi=(\mathrm{Ad}\,{u_{n}})\psi$ on $A_{n}$, $n\geq 1$. Suppose
 that there were a unitary $u\in A^{\sim}$ such that
$\varphi=\mathrm{Ad}\,{u}$. Then
$u(e_{n})\otimes u(e_{n})^{*}=\varphi(e_{n}\otimes e_{n}^{*})= e_{n+1}\otimes e_{n+1}^{*}$, $n\geq 1$.
Thus, $u(e_{n})=\lambda_{n}e_{n+1}$ for some complex number $\lambda_{n}$ with absolute value one.
Set $f_{n}=\lambda_{1}\lambda_{2}\cdots\lambda_{n-1}e_{n}$, $n\geq 1$.
Then $(f_{n})_{n=1}^{\infty}$ is an orthonormal basis for $l^2$ and
$u(f_{n})=f_{n+1}$, $n\geq 1$; in other words, $u$ is the unilateral shift, which is not a unitary.
\end{example}
\begin{corollary}\label{corun}
Let $\mathcal{A}_{1},\mathcal{A}_{2}\in\mathbf{AF}$ and
$\varphi,\psi : \mathcal{A}_{1}\to \mathcal{A}_{2}$
be  morphisms in $\mathbf{AF}$ such that
$\varphi =(\mathrm{Ad}\,{u})\psi$ for some unitary $u$
in $\mathcal{A}_{2}^{\sim}$. Then
 $\mathcal{B}(\varphi)=\mathcal{B}(\psi)$.
 \end{corollary}
\begin{corollary}\label{cordis}
Let $\varphi,\psi: \mathcal{A}_{1} \to \mathcal{A}_{2}$ morphisms  in $\mathbf{AF} $ with
$\Vert \varphi -\psi\Vert <1$. Then we have
$\mathcal{B}(\varphi)=\mathcal{B}(\psi)$.
\end{corollary}
\begin{proof}
Following the notation of Lemma~\ref{lemeqbh} and the first paragraph of its proof,
we have
$v_{n}\varphi(A_{n})v_{n}^{*}\subseteq B_{f_{n}}$ and
$w_{n}\psi(A_{n})w_{n}^{*}\subseteq B_{g_{n}}$, $n\geq 1$,
by Definition~\ref{deffu}. Fix $n\geq 1$. By Lemma~\ref{lemad},
the $*$-homomorphisms
$(\mathrm{Ad}\,{v_{n}})\varphi ,(\mathrm{Ad}\,{w_{n}})\psi : A_{n}\to B_{k_{n}}$ have the same
 multiplicity matrices, for some positive integer $k_{n}\geq f_{n},g_{n}$.
By Lemma~\ref{lemun}, there is a unitary $u\in B_{k_{n}}$ such that
$(\mathrm{Ad}\,{v_{n}})\varphi=(\mathrm{Ad}\,{u})(\mathrm{Ad}\,{w_{n}})\psi$ on $A_{n}$.
Set $w=u-1_{B_{k_{n}}}+1$, where 1 is the unit of $B^{\sim}$.
Then $w$ is a unitary in $B^{\sim}$ and again we have
$(\mathrm{Ad}\,{v_{n}})\varphi=(\mathrm{Ad}\,{w})(\mathrm{Ad}\,{w_{n}})\psi$
on $A_{n}$.
Setting $u_{n}=v_{n}^{*}ww_{n}$, we have
$\varphi=(\mathrm{Ad}\,{u_{n}})\psi$ on $A_{n}$, and so
$\mathcal{B}(\varphi)=\mathcal{B}(\psi)$ by Lemma~\ref{lemeqbh}.
\end{proof}
Consider  the category $\mathbf{AF}^{\mathrm{out}}$  associated to
$\mathbf{AF}$ as described   in \cite{el10};
its objects are the same as those of $\mathbf{AF}$ and its morphisms are
as follows. An \emph{inner automorphism} for an object
$(A, (A_{n})_{n=1}^{\infty},(\varphi_{n})_{n=1}^{\infty})$ of $\mathbf{AF}$
is a $*$-isomorphism $\mathrm{Ad}\,{u}: A\to A$, for some unitary
$u\in A^{+}$. Two morphisms $\varphi,\psi :\mathcal{A}_{1}\to \mathcal{A}_{2}$ are
\emph{equivalent} if $\varphi=(\mathrm{Ad}\,{u})\psi$ for some inner automorphism
$\mathrm{Ad}\,{u}$ of $\mathcal{A}_{2}$. Let $\mathcal{F}(\varphi)$ denote the equivalence
class of $\varphi$. These equivalence classes are the morphisms of $\mathbf{AF}^{\mathrm{out}}$.
Denote by $\mathcal{F}: \mathbf{AF}\to \mathbf{AF}^{\mathrm{out}}$
the functor which assigns to each object of $\mathbf{AF}$ itself, and maps
morphisms as above. Now  \cite[Theorem~1]{el10} states that
$\mathcal{F}: \mathbf{AF}\to \mathbf{AF}^{\mathrm{out}}$
is a  strong classification functor. Obviously, it is  also a full functor.
\begin{theorem}\label{throut}
There is a unique functor
$\widetilde{\mathcal{B}}: \mathbf{AF}^{\mathrm{out}} \to \mathbf{BD}$
such that $\mathcal{B}=\widetilde{\mathcal{B}}\,\mathcal{F}$:
\[
\xymatrix{
\mathbf{AF}\ar[r]^{\mathcal{F}}
\ar[rd]_{\mathcal{B}} & \mathbf{AF}^{\mathrm{out}}
\ar[d]^{\widetilde{\mathcal{B}}}\\
& \mathbf{BD}\ .
}
\]
Moreover, it is a strong classification functor and a full functor.
\end{theorem}
\begin{proof}
Define $\widetilde{\mathcal{B}}: \mathbf{AF}^{\mathrm{out}} \to \mathbf{BD}$
as follows. For $\mathcal{A}\in \mathbf{AF}$ set
$\widetilde{\mathcal{B}}(\mathcal{A})=\mathcal{B}(\mathcal{A})$.
Let  $\varphi :\mathcal{A}_{1}\to \mathcal{A}_{2}$ be a morphism in
$\mathbf{AF}$. By Lemma~\ref{lemeqbh} (and the remark following that),
$\mathcal{B}((\mathrm{Ad}\,{u})\varphi)=\mathcal{B}(\varphi)$,
for each inner automorphism
$\mathrm{Ad}\,{u}$ of $\mathcal{A}_{2}$. It therefore makes sense to set
$\widetilde{\mathcal{B}}(\mathcal{F}(\varphi))=\mathcal{B}(\varphi)$.
It is immediate that $\widetilde{\mathcal{B}}: \mathbf{AF}^{\mathrm{out}} \to \mathbf{BD}$
is a functor and $\mathcal{B}=\widetilde{\mathcal{B}}\,\mathcal{F}$.
Hence $\widetilde{\mathcal{B}}$ is full, since $\mathcal{B}$ is (by Theorem~\ref{thrful}).
Uniqueness follows from $\mathcal{B}=\widetilde{\mathcal{B}}\,\mathcal{F}$
and the fact that $\mathcal{F}$ is surjective on both objects and (since it is full) on maps.
That $\widetilde{\mathcal{B}}: \mathbf{AF}^{\mathrm{out}} \to \mathbf{BD}$
 is a strong classification functor follows from the fact that
 $\mathcal{B}$ is
(Theorem~\ref{thrst}).
\end{proof}
Note that the functor
$\widetilde{\mathcal{B}}: \mathbf{AF}^{\mathrm{out}} \to \mathbf{BD}$
is not faithful. For example, let $\varphi$ and $\psi$ be the morphisms in $\mathbf{AF}$
defined in Example~\ref{exaun}. Then
$\widetilde{\mathcal{B}}(\mathcal{F}(\varphi))=
\mathcal{B}(\varphi)=\mathcal{B}(\psi)=
\widetilde{\mathcal{B}}(\mathcal{F}(\psi)$, but $\mathcal{F}(\varphi)\neq \mathcal{F}(\psi)$, by
Example~\ref{exaun}. Cf.~Theorem~\ref{throutbar}.

Now let us examine the classifying category
$\overline{\mathbf{AF}^{\mathrm{out}}}$ for $\mathbf{AF}$, as described
in \cite{el10}. It is better than $\mathbf{AF}^{\mathrm{out}}$ (in some sense) for
the purposes of classification, because $\overline{\mathbf{AF}^{\mathrm{out}}}$
is  a classifying category not only for $\mathbf{AF}$, but also for
$\mathbf{AF}^{\mathrm{out}}$ (and it has even fewer automorphisms);
however, $\mathbf{BD}$ is even better than
(although, by Theorem~\ref{thrfung}, it is just equivalent to)
$\overline{\mathbf{AF}^{\mathrm{out}}}$, since it is a classifying category
for $\overline{\mathbf{AF}^{\mathrm{out}}}$ and so for
all three of these categories (by Theorem~\ref{throutbar}),
but is in some sense more explicit. (For one thing, it is a small category.)

Consider the category
$\overline{\mathbf{AF}^{\mathrm{out}}}$ as a subcategory of
$\overline{\mathcal{S}^{\mathrm{out}}}$ which is defined in \cite[Example~4.3]{el10},
where  $\mathcal{S}$ denotes the category of separable C$^*$algebras (not necessarily unital).
More precisely, the objects of $\overline{\mathbf{AF}^{\mathrm{out}}}$ are the
 same as of $\mathbf{AF}$ and its morphisms are
as follows.
For each pair of objects $\mathcal{A}_{1}$ and
$\mathcal{A}_{2}$ in $\mathbf{AF}$, and
for each $\varphi$ in
$\mathrm{Hom}(\mathcal{A}_{1},\mathcal{A}_{2})$,
denote by $\overline{\mathcal{F}}(\varphi)$
the closure of the equivalence class $\mathcal{F}(\varphi)$ in
$\mathrm{Hom}(\mathcal{A}_{1},\mathcal{A}_{2})$, in the topology of pointwise
convergence.
These are the morphisms of $\overline{\mathbf{AF}^{\mathrm{out}}}$.
By \cite[Theorem~3]{el10} and \cite[Example~4.3]{el10},
$\overline{\mathbf{AF}^{\mathrm{out}}}$ is a category.
Now define the functor
$\overline{\mathcal{F}}:\mathbf{AF}\to  \overline{\mathbf{AF}^{\mathrm{out}}}$
as follows. $\overline{\mathcal{F}}$ assigns to each object of
$\mathbf{AF}$ itself and maps morphisms as above.
By \cite[Theorem~3]{el10} and \cite[Example~4.3]{el10},
$\overline{\mathcal{F}}:\mathbf{AF}\to  \overline{\mathbf{AF}^{\mathrm{out}}}$
is a strong classification functor.
(It follows immediately that the quotient map from
$\mathbf{AF}^{\mathrm{out}}$ to
$\overline{\mathbf{AF}^{\mathrm{out}}}$ is also a
strong classification functor,
but this is not of interest to us here.)
\begin{theorem}\label{throutbar}
There is a unique functor
$\overline{\mathcal{B}}:\overline{\mathbf{AF}^{\mathrm{out}}} \to \mathbf{BD}$
such that $\mathcal{B}=\overline{\mathcal{B}}\,\overline{\mathcal{F}}$:
\[
\xymatrix{
\mathbf{AF}\ar[r]^{\overline{\mathcal{F}}}
\ar[rd]_{\mathcal{B}} & \overline{\mathbf{AF}^{\mathrm{out}}}
\ar[d]^{\overline{\mathcal{B}}}\\
& \mathbf{BD}\ .
}
\]
It is a strong classification functor, surjective on objects, and a full functor.
Moreover,  for each pair of morphisms
$\varphi,\psi :\mathcal{A}_{1}\to \mathcal{A}_{2}$ in $\mathbf{AF} $ we have
$\mathcal{B}(\varphi)=\mathcal{B}(\psi)$ if and only
 if $\overline{\mathcal{F}}(\varphi)=\overline{\mathcal{F}}(\psi)$;
 in other words,
 $\overline{\mathcal{B}}$ is a faithful functor.
\end{theorem}
\begin{proof}
First let us show that
for each pair of morphisms
$\varphi,\psi:\mathcal{A}_{1}\to \mathcal{A}_{2}$ in $\mathbf{AF} $ we have
$\mathcal{B}(\varphi)=\mathcal{B}(\psi)$ if and only
 if $\overline{\mathcal{F}}(\varphi)=\overline{\mathcal{F}}(\psi)$.
 Suppose that
$\mathcal{B}(\varphi)=\mathcal{B}(\psi)$. By Lemma~\ref{lemeqbh} and the remark
following that,
there is a sequence of unitaries
$(u_{n})_{n=1}^{\infty}$ in $B^{+}$ such that
$\varphi$ is the pointwise limit of the sequence
$((\mathrm{Ad}\,{u_{n}})\psi)_{n=1}^{\infty}$
where $B$ is the algebra (i.e., the first component)  of $\mathcal{A}_{2}$.
Thus, for each unitary $u\in B^{+}$, $(\mathrm{Ad}\,{u})\varphi$ is
the pointwise limit of the sequence
$((\mathrm{Ad}\,{uu_{n}})\psi)_{n=1}^{\infty}$.
Therefore, $\mathcal{F}(\varphi) \subseteq \overline{\mathcal{F}}(\psi)$.
Hence $\overline{\mathcal{F}}(\varphi) \subseteq \overline{\mathcal{F}}(\psi)$ and
by symmetry $\overline{\mathcal{F}}(\varphi) = \overline{\mathcal{F}}(\psi)$.
Now suppose that $\overline{\mathcal{F}}(\varphi) = \overline{\mathcal{F}}(\psi)$.
Then,
$\varphi$ is the pointwise limit of a sequence
$((\mathrm{Ad}\,{u_{n}})\psi)_{n=1}^{\infty}$, for some
sequence of unitaries
$(u_{n})_{n=1}^{\infty}$ in $B^{+}$.
By Lemma~\ref{lemeqbh},
$\mathcal{B}(\varphi)=\mathcal{B}(\psi)$.

Now define $\overline{\mathcal{B}}: \overline{\mathbf{AF}^{\mathrm{out}}} \to \mathbf{BD}$
as follows. For $\mathcal{A}\in \mathbf{AF}$ set
$\overline{\mathcal{B}}(\mathcal{A})=\mathcal{B}(\mathcal{A})$.
Let  $\varphi :\mathcal{A}_{1}\to \mathcal{A}_{2}$ be a morphism in
$\mathbf{AF}$. Set
$\overline{\mathcal{B}}(\overline{\mathcal{F}}(\varphi))=\mathcal{B}(\varphi)$.
By the preceding paragraph,
$\overline{\mathcal{B}}$ is well defined, and faithful. Also, we have
$\mathcal{B}=\overline{\mathcal{B}}\,\overline{\mathcal{F}}$. That $\overline{\mathcal{B}}$ is a functor,
and  uniqueness of $\overline{\mathcal{B}}$, follow from the fact that
$\overline{\mathcal{F}}$ is a full functor, or, rather, even surjective.
Since $\mathcal{B}$
 is a strong classification functor  and a full functor, so also is
 $\overline{\mathcal{B}}$. (That $\overline{\mathcal{B}}$ is a strong classification functor also
 follows from the fact that it is full and faithful and applying Lemma~\ref{lemcat}, below.)
\end{proof}

As we shall see, the functor
$\overline{\mathcal{B}}:\overline{\mathbf{AF}^{\mathrm{out}}} \to \mathbf{BD}$
is an equivalence of categories (see Theorem~\ref{thrfung} below).
This is mainly based on the categorical properties of this functor.
Therefore, let us first state this result in a categorical setting, in Lemma~\ref{lemcat}.
We shall use this lemma to show that the functor
$\overline{\mathcal{B}}:\overline{\mathbf{AF}^{\mathrm{out}}} \to \mathbf{BD}$
is an equivalence of categories (Theorems~\ref{thrfung}).

Recall that a functor $F : \mathcal{C} \to \mathcal{D}$ is called an equivalence
of categories if there is a functor $G : \mathcal{D} \to \mathcal{C}$ such that
$FG\cong\mathrm{id}_{\mathcal{C}}$ and
$GF\cong \mathrm{id}_{\mathcal{D}}$ \cite{gb84, ma98}.
If $H : \mathcal{D} \to \mathcal{C}$ is another functor with this property, then it is easy to see that
$H$ is naturally isomorphic to $G$. Therefore, $G$ is unique up to natural isomorphism.
It is well known that a functor $F : \mathcal{C} \to \mathcal{D}$ is an equivalence of categories
if and only if it is full, faithful, and essentially surjective, i.e., for each
$d\in \mathcal{D}$ there is a $c\in \mathcal{C}$ such that $d\cong F(c)$ \cite[Theorem~IV.4.1]{ma98}.
In the case that $F$ is surjective on objects, a right inverse for $F$ can be constructed, i.e.,
 a  functor
$G : \mathcal{D}\to \mathcal{C}$ such that $FG=\mathrm{id}_{\mathcal{C}}$ and
$GF\cong \mathrm{id}_{\mathcal{D}}$.
(A remark on the use of the axiom of choice in this context is given in the proof.)
\begin{lemma}\label{lemcat}
Let $F : \mathcal{C} \to \mathcal{D}$ be a full and faithful functor. Then $F$ is a strong classification functor.
If  $F$ is also  surjective on objects,
then it is an equivalence of categories, and, furthermore,
 there is a  unique (up to natural isomorphism) functor
$G : \mathcal{D}\to \mathcal{C}$ such that $FG=\mathrm{id}_{\mathcal{C}}$ and
$GF\cong \mathrm{id}_{\mathcal{D}}$.
The functor $G$ is full, faithful, injective on objects, essentially surjective, and
(hence) a strong classification functor.
\end{lemma}
\begin{proof}
That a full and faithful functor is a strong classification functor is straightforward.
Since  $F : \mathcal{C} \to \mathcal{D}$ is surjective on  objects,
it has a right inverse $G : \mathcal{D}\to \mathcal{C}$ (just as a map on objects).
Here we have used the ``axiom of choice" for sets or classes:
 when the objects of $\mathcal{C}$ form a set we use the axiom of choice for sets,
and when the objects of $\mathcal{C}$
form a proper class  we use the global axiom of  choice \cite{c}
(if for each object $c$ of $\mathcal{D}$ there is a canonical object $a$ in $\mathcal{C}$ such that $F(a)=c$,
one could avoid the axiom of choice).

The definition of $G$ on the morphisms of $\mathcal{D}$ was described in the proof of
 \cite[Theorem~IV.4.1]{ma98}, and so we have a functor $G$ such that
$FG=\mathrm{id}_{\mathcal{C}}$ and
$GF\cong \mathrm{id}_{\mathcal{D}}$.
 The rest follows from the fact
 that each functor which is an equivalence of categories is full, faithful, and essentially surjective
\cite[Theorem~IV.4.1]{ma98}.
\end{proof}
The following theorem states that the categories
$\overline{\mathbf{AF}^{\mathrm{out}}}$ and $\mathbf{BD}$ are equivalent, and this equivalence,
given by the functor
$\overline{\mathcal{B}}:\overline{\mathbf{AF}^{\mathrm{out}}} \to \mathbf{BD}$,
 is compatible with the  classification of AF~algebras via the functors
$\mathcal{B}: \mathbf{AF} \to \mathbf{BD}$ and
$\overline{\mathcal{F}}:\mathbf{AF}\to  \overline{\mathbf{AF}^{\mathrm{out}}}$, i.e.,
the related diagrams commute.
\begin{theorem}\label{thrfung}
The functor $\overline{\mathcal{B}}:\overline{\mathbf{AF}^{\mathrm{out}}} \to \mathbf{BD}$
is an equivalence of categories. More precisely,
there is a unique (up to natural isomorphism) functor
$\mathcal{G}:\mathbf{BD}\to \overline{\mathbf{AF}^{\mathrm{out}}}$
such that
$\overline{\mathcal{B}}\,\mathcal{G}=\mathrm{id}_{\mathbf{BD}}$ and
$\mathcal{G}\overline{\mathcal{B}}\cong \mathrm{id}_{\overline{\mathbf{AF}^{\mathrm{out}}}}$.
The functor $\mathcal{G}$ is
full, faithful, injective on objects, essentially surjective, and a strong classification  functor.
Moreover, for each $B,C\in \mathbf{BD}$ and each morphism
$\varphi :\mathcal{G}(B)\to \mathcal{G}(C)$ in $\mathbf{AF}$, we have
$\mathcal{G}\mathcal{B}(\varphi)=\overline{\mathcal{F}}(\varphi)$, i.e.,
the following diagram commutes:

\[
\xymatrix{
\mathbf{AF}\ar[r]^{\overline{\mathcal{F}}}
\ar[d]_{\mathcal{B}} & \overline{\mathbf{AF}^{\mathrm{out}}}
\ar[d]^{\overline{\mathcal{B}}}\\
\mathbf{BD} \ar[r]_{\mathrm{id}}\ar[ur]^{\mathcal{G}}& \mathbf{BD}\ .
}
\]

\end{theorem}
\begin{proof}
By Theorem~\ref{throutbar}, the functor $\overline{\mathcal{B}}:\overline{\mathbf{AF}^{\mathrm{out}}} \to \mathbf{BD}$ is full, faithful, surjective on objects, and a strong classification functor. By
Lemma~\ref{lemcat}, it is also an equivalence of categories and the functor
$\mathcal{G}:\mathbf{BD}\to \overline{\mathbf{AF}^{\mathrm{out}}}$ with the desired properties exists.
As indicated in the proof of Lemma~\ref{lemcat}, here the use of the axiom of choice is justified as follows. Since
the collection of the objects $\overline{\mathbf{AF}^{\mathrm{out}}}$ is a proper class and
 $\mathbf{BD}$ is a small category, we can use the global axiom of  choice \cite{c}.
Alternatively, one could use the fact that each AF~algebra is (isomorphic to) a C$^*$-subalgebra of
$\mathbb{B}(l^{2})$, and essentially
the axiom of choice for sets is enough.
Finally, one  can choose $\mathcal{G}(B)$   to be the  AF~algebra constructed as in
Proposition~\ref{proonto}.

For the last statement, let $B,C\in \mathbf{BD}$ and
$\varphi :\mathcal{G}(B)\to \mathcal{G}(C)$ be a morphism in $\mathbf{AF}$. Note that the objects of
$\mathbf{AF}$ and $\overline{\mathbf{AF}^{\mathrm{out}}}$ are the same, and so
$\mathcal{G}(B)$ and  $\mathcal{G}(C)$ are also in
$\mathbf{AF}$. We have
$\mathcal{G}\mathcal{B}(\mathcal{G}(B))=
\mathcal{G}\overline{\mathcal{B}}(\mathcal{G}(B))=
\mathcal{G}(B)=\overline{\mathcal{F}}(\mathcal{G}(B))$, and similarly for $C$.
Thus,
$\mathcal{G}\mathcal{B}(\varphi)$ and $\overline{\mathcal{F}}(\varphi)$
have the same domains and the same ranges. We have
$\overline{\mathcal{B}}(\mathcal{G}\mathcal{B}(\varphi))=
\mathcal{B}(\varphi)=\overline{\mathcal{B}}(\overline{\mathcal{F}}(\varphi))$,
since $\mathcal{B}=\overline{\mathcal{B}}\,\overline{\mathcal{F}}$. By
Theorem~\ref{throutbar}, $\overline{\mathcal{B}}$ is faithful, and so
$\mathcal{G}\mathcal{B}(\varphi)=\overline{\mathcal{F}}(\varphi)$.
\end{proof}

\section{Relations with K-Theory}
\noindent Consider the category $\mathbf{DG}$  of dimension groups, i.e.,
the set of all scaled countable ordered groups which are unperforated  and have
the Riesz decomposition property,
 with order and scale preserving homomorphisms (see \cite{bl98, ehs80, wo93}), and consider
the well-known
$\mathrm{K}_{0}$~functor $\mathrm{K}_{0}: \mathbf{AF}\to \mathbf{DG}$.
The following statement summarizes the main properties of the functor
$\mathrm{K}_{0}: \mathbf{AF}\to \mathbf{DG}$.
\begin{theorem}\label{threehs}
The functor $\mathrm{K}_{0}: \mathbf{AF}\to \mathbf{DG}$ is a strong classification functor and a full functor.
Moreover, it is essentially surjective on objects.
\end{theorem}
\begin{proof}
That the functor $\mathrm{K}_{0}: \mathbf{AF}\to \mathbf{DG}$ is
a strong classification functor is  Elliott's theorem, \cite{el76}. That the functor
$\mathrm{K}_{0}: \mathbf{AF}\to \mathbf{DG}$ is full is known,
and the proof is similar to the proof that it is  a strong classification functor---one
uses a one-sided intertwining argument rather than a two-sided one,
just as in Theorem~\ref{thrful}.
In fact, one can deduce it from
Theorem~\ref{thrful}
together
with the factorization of $\mathrm{K}_{0}: \mathbf{AF}\to \mathbf{DG}$
 through $\mathbf{BD}$ by means of the
inductive limit functor described in the alternative proof of Corollary~\ref{corbddg}
 below, which is easily seen to be full, and so $\mathrm{K}_{0}$ is expressed as the
composition of two full functors.
 The last statement follows from the Effros-Handelman-Shen theorem
\cite[Theorem~2.2]{ehs80}, and the result of Elliott, \cite[Theorem~5.5]{el76} characterizing  $\mathrm{K}_{0}$~groups of AF~algebras as inductive limits.
\end{proof}
The following lemma is surely part of the literature;  we give a proof for the sake of completeness. We follow \cite{wo93} for  K-theory notation.
\begin{lemma}\label{lemeqk}
Let $\mathcal{A}_{1},\mathcal{A}_{2}\in\mathbf{AF}$ where
$\mathcal{A}_{1}=(A, (A_{n})_{n=1}^{\infty},(\varphi_{n})_{n=1}^{\infty})$
and
$\mathcal{A}_{2}=(B, (B_{n})_{n=1}^{\infty},(\psi_{n})_{n=1}^{\infty})$.
Let
$\varphi , \psi : \mathcal{A}_{1}\to \mathcal{A}_{2}$
be  morphisms in $\mathbf{AF}$. The following statements are equivalent:
\begin{itemize}
\item[(1)]
$\mathrm{K}_{0}(\varphi)=\mathrm{K}_{0}(\psi)$;
\item[(2)]
there is a sequence of unitaries
$(u_{n})_{n=1}^{\infty}$ in $B^{\sim}$ such that
$\varphi=(\mathrm{Ad}\,{u_{n}})\psi$ on $A_{n}$, $n\geq 1$.
\end{itemize}
\end{lemma}
\begin{proof}
Let $\varphi^{+} , \psi^{+} :A^{+} \to B^{+}$ denote the unital extensions of
$\varphi , \psi : A\to B$. Note that
$A^{+}=\overline{\bigcup_{n=1}^{\infty}A_{n}^{+}}$ and
$B^{+}=\overline{\bigcup_{n=1}^{\infty}B_{n}^{+}}$.
The proof  is similar to the finite dimensional case \cite[Theorem~7.2.6]{mu90}.

$(1)\Rightarrow (2)$:
The proof really should be thought of as three
separate steps---first reducing to the case that the domain
is a single finite dimensional algebra, and then to the case
that the codomain is a single finite dimensional algebra (using
for the second step that $\mathrm{K}_{0}$ of the limit is the limit of the
$\mathrm{K}_{0}$'s). The third step, that both algebras are finite dimensional,
 follows immediately from an argument due to Bratteli.
The details are as follows.

Fix $n\geq 1$. By Lemma~\ref{lembr}, there are unitaries
$u,v\in B^{+}$ and a positive integer $m$ such that
$u\varphi^{+}(A_{n}^{+})u^{*}\subseteq B_{m}^{+}$ and
$v\psi^{+}(A_{n}^{+})v^{*}\subseteq B_{m}^{+}$.
Define $*$-homomorphisms $F,G : A_{n}^{+}\to B_{m}^{+}$ with $F(a)=u\varphi^{+}(a)u^{*}$ and
$G(a)=v\psi^{+}(a)v^{*}$, $a\in A_{n}^{+}$.
Let $\{e_{ij}^{l} : 1\leq l\leq k, \ 1\leq i,j\leq n_{l}\}$ be a set of matrix units for $A_{n}$
and set $e_{11}^{k+1}=1-1_{A_{n}}$ (where 1 is the unit element of $ A_{n}^{+}$) and $n_{k+1}=1$. Then
$\{e_{ij}^{l} : 1\leq l\leq k+1, \ 1\leq i,j\leq n_{l}\}$ is a set of matrix units for $A_{n}^{+}$.

Set
$p_{ij}^{l}=F(e_{ij}^{l})$ and
$q_{ij}^{l}=G(e_{ij}^{l})$, for the above values of $i,j,l$.
Fix $1\leq l\leq k+1$.
Let $[p_{11}^{l}]$ be the equivalence class of the projection $p_{11}^{l}$ in $V(B^{+})$.
Since $e_{11}^{l}\in A$, for $1\leq l\leq k$, the formal difference $x=[e_{11}^{l}]-[0]$ is in $\mathrm{K}_{0}(A)$.
We have
$[p_{11}^{l}]-[0]=[\varphi(e_{11}^{l})]-[0]=\mathrm{K}_{0}(\varphi)(x)=
\mathrm{K}_{0}(\psi)(x)=[q_{11}^{l}]-[0]$. Thus there is $r\in \mathbb{M}_{\infty}(B^{+})$ such that
$[p_{11}^{l}]+[r]=[q_{11}^{l}]+[r]$ in $V(B^{+})$.
For $l=k+1$ we have $[e_{11}^{k+1}]-[1]\in \mathrm{K}_{0}(A)$ and similarly
there is $r\in \mathbb{M}_{\infty}(B^{+})$ such that
$[p_{11}^{k+1}]+[r]=[q_{11}^{k+1}]+[r]$ in $V(B^{+})$.
Since $B^{+}=\overline{\bigcup_{n=1}^{\infty}B_{n}^{+}}$, there is $m'\geq m$ such that
$r$ is equivalent to some projection in $\mathbb{M}_{\infty}(B_{m'}^{+})$, and we may
assume that $r\in \mathbb{M}_{\infty}(B_{m'}^{+})$.
Hence the projections
$\mathrm{diag}(p_{11}^{l},r),\mathrm{diag}(q_{11}^{l},r)\in \mathbb{M}_{\infty}(B_{m'}^{+})$
are equivalent in $B^{+}$. By \cite[Lemma~7.2.8]{mu90} (or, rather, by its proof),
there is $m_{l}\geq m'$ such that
$\mathrm{diag}(p_{11}^{l},r)$ and $\mathrm{diag}(q_{11}^{l},r)$ are equivalent in
$B_{m_{l}}^{+}$. Since $B_{m_{l}}^{+}$ is finite dimensional,
$p_{11}^{l}$ is  equivalent to $q_{11}^{l}$, and so there is
a partial isometry  $w_{l}\in B_{m_{l}}^{+}$ such that $p_{11}^{l}=w_{l}w_{l}^{*}$
and $q_{11}^{l}=w_{l}^{*}w_{l}$. Now set
\[
\sum_{l=1}^{k+1}\sum_{i=1}^{n_{l}}p_{i1}^{l}w_{l}q_{1i}^{l}=w.
\]
An easy  calculation shows that $w$ is unitary in $B^{+}$ and $wq_{ij}^{l}=p_{ij}^{l}w$, for the above values of $i,j,l$.
Thus $F=(\mathrm{Ad}\,{w})G$. Set $u_{n}=u^{*}wv$. Then
$\varphi=(\mathrm{Ad}\,{u_{n}})\psi$ on $A_{n}$. We have found a sequence of
unitaries
$(u_{n})_{n=1}^{\infty}$ in $B^{+}$ with the property stated in part (2), but this is
equivalent to the existence of a sequence of unitaries in $B^{\sim}$ with the same property,
by the remark following Lemma~\ref{lemeqbh}.

$(2)\Rightarrow (1)$: Suppose that (2) holds. As stated above, we may assume that
the unitaries
$(u_{n})_{n=1}^{\infty}$ are in $B^{+}$. Let $p$ be in
$\mathbb{M}_{\infty}(A^{+})$. Then there is a positive integer $n$ and a projection
$q\in \mathbb{M}_{\infty}(A_{n}^{+})$ such that $[p]=[q]$.
Since $\varphi=(\mathrm{Ad}\,{u_{n}})\psi$ on $A_{n}$, and we have
$\varphi^{+}=(\mathrm{Ad}\,{u_{n}})\psi^{+}$ on $A_{n}^{+}$,  hence
$\varphi_{*}([p])=\varphi_{*}([q])=\psi_{*}([q])=\psi_{*}([p])$. Therefore
$\mathrm{K}_{0}(\varphi)=\mathrm{K}_{0}(\psi)$.
\end{proof}
\begin{theorem}\label{thrkbar}
There is a unique functor
$\overline{\mathrm{K}}_{0}:\overline{\mathbf{AF}^{\mathrm{out}}} \to \mathbf{DG}$
such that $\mathrm{K}_{0}=\overline{\mathrm{K}}_{0}\,\overline{\mathcal{F}}$:
\[
\xymatrix{
\mathbf{AF}\ar[r]^{\overline{\mathcal{F}}}
\ar[rd]_{\mathrm{K}_{0}} & \overline{\mathbf{AF}^{\mathrm{out}}}
\ar[d]^{\overline{\mathrm{K}}_{0}}\\
& \mathbf{DG}\ .
}
\]
It is a strong classification functor, essentially surjective on objects, and a full functor.
 Moreover,  for each pair of morphisms
$\varphi,\psi :\mathcal{A}_{1}\to \mathcal{A}_{2}$ in $\mathbf{AF} $, we have
$\mathrm{K}_{0}(\varphi)=\mathrm{K}_{0}(\psi)$ if and only
 if $\overline{\mathcal{F}}(\varphi)=\overline{\mathcal{F}}(\psi)$; in  other words,
 $\overline{\mathrm{K}}_{0}$ is a faithful functor.
\end{theorem}
\begin{proof}
Define $\overline{\mathrm{K}}_{0}:\overline{\mathbf{AF}^{\mathrm{out}}} \to \mathbf{DG}$ as follows.
For
$\mathcal{A}\in \overline{\mathbf{AF}^{\mathrm{out}}}$ set
$\overline{\mathrm{K}}_{0}(\mathcal{A})=\mathrm{K}_{0}(\mathcal{A})$.
Let  $\varphi :\mathcal{A}_{1}\to \mathcal{A}_{2}$ be a morphism in
$\mathbf{AF}$. Set
$\overline{\mathrm{K}}_{0}(\overline{\mathcal{F}}(\varphi))=\mathrm{K}_{0}(\varphi)$.
For each pair of morphisms
$\varphi,\psi :\mathcal{A}_{1}\to \mathcal{A}_{2}$ in $\mathbf{AF}$ we have
$\mathrm{K}_{0}(\varphi)=\mathrm{K}_{0}(\psi)$ if and only
 if $\overline{\mathcal{F}}(\varphi)=\overline{\mathcal{F}}(\psi)$, by
 Lemma~\ref{lemeqk}.
This proves that $\overline{\mathrm{K}}_{0}$ is well~defined. The other properties of
 $\overline{\mathrm{K}}_{0}$ are easy to prove, using Theorem~\ref{threehs} and the same properties of
$\overline{\mathcal{F}}$ as  used in
the proof of Theorem~\ref{throutbar}.

Alternatively, the whole theorem follows immediately from
Theorem~\ref{throutbar} together with the equivalence of the categories
$\mathbf{BD}$ and $\mathbf{DG}$,
an elementary proof of which is given below as an alternative proof of
Corollary~\ref{corbddg}---one uses that the functor
$\mathbf{BD}\to \mathbf{DG}$ in question is just the inductive limit
functor, which acts as a natural isomorphism between the functors
$\mathcal{B}$ and $\mathrm{K}_{0}$,
and therefore also between $\overline{\mathcal{B}}$
and $\overline{\mathrm{K}}_{0}$.
\end{proof}

Next let us show that the categories
$\overline{\mathbf{AF}^{\mathrm{out}}}$ and $\mathbf{DG}$ are equivalent, and
this equivalence, given by
the functor
$\overline{\mathrm{K}}_{0}:\overline{\mathbf{AF}^{\mathrm{out}}} \to \mathbf{DG}$,
 is compatible with the classification of AF~algebras via the functors
$\mathcal{B}: \mathbf{AF} \to \mathbf{BD}$ and
$\overline{\mathcal{F}}:\mathbf{AF}\to  \overline{\mathbf{AF}^{\mathrm{out}}}$, i.e.,
the related diagrams commute.
\begin{theorem}\label{thrfung0}
The functor $\overline{\mathrm{K}}_{0}:\overline{\mathbf{AF}^{\mathrm{out}}} \to \mathbf{DG}$
is an equivalence of categories. More precisely,
there is a  unique (up to natural isomorphism) functor
$\mathcal{G}_{0}:\mathbf{DG}\to \overline{\mathbf{AF}^{\mathrm{out}}}$
such that
$\overline{\mathrm{K}}_{0}\,\mathcal{G}_{0}\cong\mathrm{id}_{\mathbf{DG}}$ and
$\mathcal{G}_{0}\overline{\mathrm{K}}_{0}\cong \mathrm{id}_{\overline{\mathbf{AF}^{\mathrm{out}}}}$.
The functor $\mathcal{G}_{0}$ is
full, faithful, essentially surjective, and a strong classification  functor.
Moreover,
we have
$\overline{\mathrm{K}}_{0}\mathcal{G}_{0}{\mathrm{K}}_{0}
 ={\mathrm{K}}_{0}$:

\[
\xymatrix{
\mathbf{AF}\ar[r]^{{\mathrm{K}}_{0}}
\ar[d]_{{\mathrm{K}}_{0}} &  \mathbf{DG}\\
\mathbf{DG} \ar[r]_{\mathcal{G}_{0}}
&  \overline{\mathbf{AF}^{\mathrm{out}}}
\ar[u]_{\overline{\mathrm{K}}_{0}}\ .
}
\]

\end{theorem}
\begin{proof}
The first part of the statement follows from
 Theorem~\ref{thrkbar} and \cite[Theorem~IV.4.1]{ma98}.
The functorial properties of
$\mathcal{G}_{0}$ follow from
\cite[Theorem~IV.4.1]{ma98} and
 Lemma~\ref{lemcat}.
 For the proof  of the last part of the statement
 note that we can construct $\mathcal{G}_{0}$
 to have the property that
 $\overline{\mathrm{K}}_{0}(\mathcal{G}_{0}(\overline{\mathrm{K}}_{0}(\mathcal{A})))
 =\overline{\mathrm{K}}_{0}(\mathcal{A})$,
 for each $\mathcal{A}\in \overline{\mathbf{AF}^{\mathrm{out}}}$,
 and that
 $\overline{\mathrm{K}}_{0}(\mathcal{G}_{0}(\overline{\mathrm{K}}_{0}(\psi)))
 =\overline{\mathrm{K}}_{0}(\psi)$,
for each morphism $\psi :\mathcal{A}_{1}\to \mathcal{A}_{2}$
in $\overline{\mathbf{AF}^{\mathrm{out}}}$
(see the proof of \cite[Theorem~IV.4.1]{ma98}).
Therefore, for each morphism
 $\varphi :\mathcal{A}_{1}\to \mathcal{A}_{2}$ in
 $\mathbf{AF}$,
 setting $\psi=\overline{\mathcal{F}}(\varphi)$
 we get
$\overline{\mathrm{K}}_{0}(\mathcal{G}_{0}({\mathrm{K}}_{0}(\varphi)))
 ={\mathrm{K}}_{0}(\varphi)$.

Alternatively, use the remarks
at the end of the proof of Theorem~\ref{thrkbar}.
\end{proof}
The category $\overline{\mathbf{AF}^{\mathrm{out}}}$
can be used to relate the categories
$\mathbf{BD}$ and $\mathbf{DG}$. Let us collect all the functors
in one (commutative) diagram:

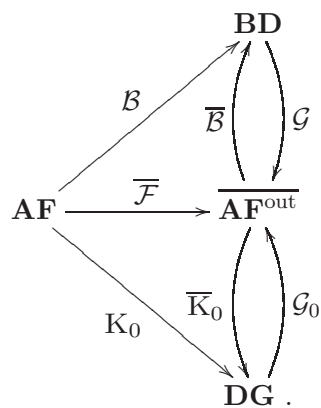
\[
\xymatrix{
& & \mathbf{BD}
\ar@/^.8pc/[dd]^{\resizebox{.65em}{.65em}{$\mathcal{G}$}}\\
& & \\
\mathbf{AF}\ar[rr]^{\resizebox{.85em}{.85em}{$\overline{\mathcal{F}}$}}
\ar[rrdd]_{\resizebox{1.3em}{.67em}{$\mathrm{K}_{0}$}}
\ar[rruu]^{\resizebox{.66em}{.66em}{$\mathcal{B}$}} & &
\overline{\mathbf{AF}^{\mathrm{out}}}
\ar@/_.8pc/[dd]_{\resizebox{1.29em}{.84em}{$\overline{\mathrm{K}}_{0}$}}
\ar@/^.8pc/[uu]^{\resizebox{.65em}{.85em}{$\overline{\mathcal{B}}$}}\\
& &\\
& & \mathbf{DG}
\ar@/_.8pc/[uu]_{\resizebox{1em}{.65em}{$\mathcal{G}_{0}$}}\ .
}
\]

\begin{corollary}\label{corbddg}
There is a full, faithful, essentially surjective, and  (hence) strong classification functor
from $\mathbf{BD}$ to
$\mathbf{DG}$. Also there is such a functor from
$\mathbf{DG}$ to $\mathbf{BD}$.
\end{corollary}
\begin{proof}
Consider the functors
$\mathcal{G}:\mathbf{BD}\to \overline{\mathbf{AF}^{\mathrm{out}}}$ and
$\overline{\mathrm{K}}_{0}:\overline{\mathbf{AF}^{\mathrm{out}}} \to \mathbf{DG}$. Both of these
are strong classification functors and also full, faithful and essentially surjective,
by Theorems~\ref{thrfung} and \ref{thrkbar}, and so also is
$\overline{\mathrm{K}}_{0}\,\mathcal{G}:\mathbf{BD}\to \mathbf{DG}$.
Similarly, the functor $\overline{\mathcal{B}}\,\mathcal{G}_{0}:\mathbf{DG}\to \mathbf{BD}$
has the above properties, by Theorems~\ref{throutbar} and \ref{thrfung0}.

Alternatively, and in a much more elementary way, the obvious inductive
limit functor $\mathbf{BD} \to \mathbf{DG}$,
obtained by interpreting a Bratteli diagram as a sequence of finite
ordered group direct sums of copies of $\mathbb{Z}$,
as in \cite[Section~2]{el10}, is full, faithful, and essentially surjective by
\cite[Theorem~2.2]{ehs80}, and so by \cite[Theorem~IV.4.1]{ma98}
an equivalence of categories.
\end{proof}
\begin{corollary}
The categories $\mathbf{BD}$,  $\overline{\mathbf{AF}^{\mathrm{out}}}$, and
$\mathbf{DG}$ are equivalent.
\end{corollary}
Finally let us show that the three strong classification functors $\mathcal{B}:\mathbf{AF}\to \mathbf{BD}$,
$\overline{\mathcal{F}} :\mathbf{AF}\to  \overline{\mathbf{AF}^{\mathrm{out}}}$, and
$\mathrm{K}_{0}: \mathbf{AF}\to \mathbf{DG}$ for classification
of AF~algebras are essentially the same.
\begin{theorem}\label{thrfuns}
Consider the three strong classification functors
$\mathcal{B}:\mathbf{AF}\to \mathbf{BD}$,
$\overline{\mathcal{F}} :\mathbf{AF}\to  \overline{\mathbf{AF}^{\mathrm{out}}}$, and
$\mathrm{K}_{0}: \mathbf{AF}\to \mathbf{DG}$ for AF~algebras. For each objects $\mathcal{A}_{1},\mathcal{A}_{2}$
in $\mathbf{AF}$, the following are equivalent:
\begin{itemize}
\item[(1)]
$\mathcal{A}_{1}\cong\mathcal{A}_{2}$
in $\mathbf{AF}$;
\item[(2)]
$\mathcal{B}(\mathcal{A}_{1})\cong\mathcal{B}(\mathcal{A}_{2})$
in $\mathbf{BD}$;
\item[(3)]
$\overline{\mathcal{F}}(\mathcal{A}_{1})\cong
\overline{\mathcal{F}}(\mathcal{A}_{2})$
in $\overline{\mathbf{AF}^{\mathrm{out}}}$;
\item[(4)]
$\mathrm{K}_{0}(\mathcal{A}_{1})\cong \mathrm{K}_{0}(\mathcal{A}_{2})$
in $\mathbf{DG}$.
\end{itemize}
For each pair of morphisms
$\varphi ,\psi:\mathcal{A}_{1}\to \mathcal{A}_{2}$ in $\mathbf{AF}$,
the following statements are equivalent:
\begin{itemize}
\item[(1)]
$\mathcal{B}(\varphi)=\mathcal{B}(\psi)$;
\item[(2)]
$\overline{\mathcal{F}}(\varphi)=\overline{\mathcal{F}}(\psi)$;
\item[(3)]
$\mathrm{K}_{0}(\varphi)=\mathrm{K}_{0}(\psi)$;
\item[(4)]
there is a sequence of unitaries
$(u_{n})_{n=1}^{\infty}$ in $B^{\sim}$ such that
$\varphi=(\mathrm{Ad}\,{u_{n}})\psi$ on $A_{n}$, $n\geq 1$;
\item[(5)]
there is a sequence of unitaries
$(u_{n})_{n=1}^{\infty}$ in $B^{\sim}$ such that
$\varphi(a)=\lim\limits_{n\to \infty}(\mathrm{Ad}\,{u_{n}})\psi(a)$, $a\in A$.
\end{itemize}
\end{theorem}
\begin{proof}
The first part follows from the fact that we are considering  classification functors.
The second part follows from
Lemma~\ref{lemeqbh}, Theorem~\ref{throutbar}, and
Lemma~\ref{lemeqk}.
\end{proof}

\end{document}